\documentclass{article}
\usepackage[utf8]{inputenc}
\usepackage{amssymb}
\usepackage{amsfonts}
\usepackage{graphicx}
\usepackage{amsmath}
\usepackage{color}
\usepackage{stmaryrd} 
\usepackage[colorlinks=true,urlcolor=red, pdftitle=On the global asymptotic stability of an infection-age structured competitive model]{hyperref}
\usepackage{cleveref}
\usepackage{amsthm} 
\usepackage[comma,numbers,square,sort&compress]{natbib} 
\usepackage{geometry} 
\usepackage{enumitem} 
\usepackage{authblk} 

\newcommand{\RR}{\mathcal{R}} 
\newcommand{\R}{\mathbb{R}} 
\newcommand{\N}{\mathbb{N}} 
\newcommand{\supp}{\text{supp}} 
\newcommand{\X}{\mathcal{X}} 
\newcommand{\B}{\mathcal{B}}
\renewcommand{\L}{\mathcal{L}}
\newcommand{\Co}{\mathcal{C}} 
\newcommand{\A}{\mathcal{A}} 
\renewcommand{\S}{\mathcal{S}}
\newcommand{\ep}{\varepsilon} 
\newcommand{\V}{\mathcal{V}} 
\newcommand{\W}{\mathcal{W}}
\newcommand{\M}{\mathcal{M}}
\newcommand{\E}{\mathcal{E}^*} 
\newcommand{\J}{\mathcal{J}}
\renewcommand{\H}{\mathcal{H}}
\newcommand{\p}{\mathbf{p}}

\newtheorem{theorem}{Theorem}[section]
\newtheorem{proposition}[theorem]{Proposition}
\newtheorem{definition}[theorem]{Definition}
\newtheorem{assumption}[theorem]{Assumption}
\newtheorem{corollary}[theorem]{Corollary}

\newtheorem{lemma}[theorem]{Lemma}

\usepackage[dvipsnames]{xcolor} 

\begin{document}

\title{On the global asymptotic stability of an infection-age structured competitive model}
\date{}
\author[1]{S.~Girel}
\affil[1]{Université Côte d’Azur, CNRS, LJAD, Parc Valrose, 06108 Nice, France}
		
\author[2]{Q.~Richard}
\affil[2]{Université de Montpellier, CNRS, IMAG, 34090 Montpellier, France.}

\maketitle

\hrule

\begin{abstract}
We investigate an infection-age structured competitive epidemiological model involving multiple strains. While classical results establish competitive exclusion when a unique maximal basic reproduction number exists, we provide here a complete characterization of the asymptotic behavior for an arbitrary number of populations without assuming uniqueness of the maximal reproduction number. By means of integrated semigroups theory, persistence results, and Lyapunov functionals, we establish global asymptotic stability of equilibria and extend previous results obtained for simpler (ODE) models. A key contribution lies in overcoming technical difficulties related to the definition and differentiation of Lyapunov functionals, as well as in refining arguments based on the LaSalle invariance principle.
\end{abstract}

\textit{Keywords:} Infection-age structured model; Competitive exclusion principle; Global asymptotic stability; Lyapunov functionals; Persistence theory; Integrated semigroups theory
\vspace{0.3cm}

\hrule

\section{Introduction}

In this paper we focus on the following infection-age structured competitive model
\begin{equation}\label{Eq:model}
    \left\{
    \begin{array}{rcl}
    S'(t)&=&\Lambda-\mu_SS(t)-S(t)\displaystyle\sum_{k=1}^n\int_0^\infty \beta_k(a)x_k(t,a)da,\\
    \dfrac{\partial x_k}{\partial t}(t,a)+\dfrac{\partial x_k}{\partial a}(t,a)&=&-\mu_k(a)x_k(t,a),\\
    x_k(t,0)&=&S(t)\displaystyle\int_0^\infty \beta_k(a)x_k(t,a)da \\
    (S(0),x_1(0,.),...,x_n(0,.))&=&(S_0,x_{1,0},...,x_{n,0})\in \R_+\times (L^1_+(0,\infty))^n
    \end{array}
    \right.
\end{equation}
for every $(t,a,k)\in (\R_+^*)^2\times\llbracket 1,n \rrbracket$ where $n\geq 1$ is the number of infected populations. The case $n=1$ was investigated by Thieme and Castillo-Chavez in \cite{ThiemeCC89,ThiemeCC93} with the study of the uniform persistence of the system and the local asymptotic stability (LAS) of the endemic equilibrium. The analysis was completed by Magal, McCluskey and Webb who established the global asymptotic stability (GAS) of both the disease-free and the endemic equilibria \cite{MagCluskWebb2010}. Using the next generation operator approach \cite{Diekmann90} they derive an explicit expression for the basic reproduction number $\RR_0$ and prove that if $\RR_0\leq 1$ then the disease-free equilibrium is GAS in the whole space, while if $\RR_0>1$ then the unique endemic equilibrium is GAS in the set of initial conditions for which disease transmission occurs, and the disease-free equilibrium is GAS in the complementary set. Later, the case $n\geq 2$ was investigated by Martcheva and Li using a more general model that includes a vaccinated compartment \cite{MartchevaLi2013}. They derive a basic reproduction number $\RR_{0,k}$ for each species $k\in\llbracket 1,n\rrbracket$ and prove that if $\max_{k\in \llbracket 1,n\rrbracket}\{\RR_{0,k}\}\leq 1$ then the disease-free equilibrium is GAS in the whole space. Moreover, under the assumption that there exists $j\in\llbracket 1,n\rrbracket$ such that $\RR_{0,j}>1$ and $\RR_{0,j}>\max_{k\in\llbracket 1,n\rrbracket\setminus \{j\}}\{\RR_{0,k}\}$ then the endemic equilibrium corresponding to the survival of the strain $j$ only $(i^*_j\neq 0$ and $i^*_k=0$ if $k\neq j)$, is GAS in the set of initial conditions for which the transmission of the disease $j$ occurs. This classical result is commonly called competitive exclusion principle and was postulated first by Gause in an ecological context \cite{Gause34} and later reinterpreted in an epidemiological context by Bremermann and Thieme \cite{BremermannThieme89}. It states that $n$ competitors cannot coexist on less than $n$ resources, hence if multiple strains circulate in the population then only
the strain with the largest reproduction number persists, while the strains with suboptimal reproduction numbers are eliminated. As mentioned by the authors, the case where the maximal reproduction number is reached for multiple populations is not treated in \cite{MartchevaLi2013}. For two species $(n=2)$, this case was partly investigated in \cite{Richard2020} with an attractiveness result of the set composed of all endemic equilibria. In the absence of age-infection structure, the ODE model has been extensively studied (see \cite{Hsu78,Hsu77} for Holling type 2 models or more recently \cite{BDG23} with infinitely many variants). The asymptotic behavior of the solutions is fully characterized in all cases. Moreover, when the maximal $\RR_0$ is attained for multiple populations, the equilibrium toward which the solution converges can be explicitly determined based on the initial condition. This issue, however, remains open in the PDE setting.

The methodology employed in \cite{MagCluskWebb2010} consists first in using integrated semigroup theory (see \cite{MagalRuan2018}) to get an appropriate framework for the well-posedness of the problem. It allows to linearize the system around each equilibrium, obtain linear $\mathcal{C}_0$-semigroups, then use spectral theory to obtain the local stability of the equilibria (see \cite{EngelNagel2000,Webb85}). Afterwards the study focuses on the semiflow (see \cite{Hale1988} for definitions and results) by proving that it is both bounded-dissipative and asymptotically smooth, which implies the existence of a strong global attractor $\A$ that is LAS. The last part of the proof consists in showing that this attractor reduces to some equilibrium, by using Lyapunov functionals. The same kind of arguments are used in \cite{MartchevaLi2013}. More recently, global stability results were established in a SI epidemiological PDE model with infection load-structure \cite{Perasso2019}. This is also done by using Lyapunov functionals, however, contrarily to \cite{MagCluskWebb2010,MartchevaLi2013}, the computations are not done on the global attractor but on the omega-limit set for each initial condition. This method comes at the cost of requiring local asymptotic stability to be established independently.

In order to reduce the attractor to some equilibrium, we first must check that the Lyapunov candidate function is well defined on the attractor. This is immediate for the disease-free equilibrium, whereas strong uniform persistence results are necessary for the endemic equilibrium. In \cite{MagCluskWebb2010,MartchevaLi2013}, the weak persistence of the semiflow is readily obtained, then the uniform weak persistence is deduced from the tools developed in \cite{HaleWaltman89} (see also \cite{SmithThieme2011}). This ultimately leads to the uniform strong persistence due to the existence of a global attractor \cite{SmithThieme2011}. In \cite{Perasso2019} the author combines the weak persistence with the fact that the disease-free equilibrium is included in the omega-limit set to claim a contradiction, but there may be cases where the solution approaches the disease-free equilibrium without converging to it. This proof was unfortunately reused in \cite{Richard2020}.

The second point is to compute the derivative of the Lyapunov candidate function. As explained in \cite{MagCluskWebb2010}, initial conditions in some specific domain lead to classical solutions \cite{MagalRuan2018} of \eqref{Eq:model} and the PDE equations can be used to readily compute the derivative. For other initial conditions, the key argument is the density of the latter domain into the positive cone of the whole set as mentioned in \cite{MagCluskWebb2010}. Though it is also necessary that the semiflow is state-continuous uniformly in finite time (see \cite{SmithThieme2011}) which naturally follows from the Lipschitz property of the non-linearity. For the endemic equilibrium, this is even trickier since the sequence of initial conditions obtained by the density argument does not necessarily satisfy the estimates required for the Lyapunov candidate function to be well-defined. Hence, the construction of an \textit{ad hoc} sequence, which both belong to the domain and satisfy the estimates, is necessary. This is done in the present paper.

The third and last point is the reduction of the attractor to the equilibrium. In \cite{MagCluskWebb2010}, this is done by considering a point $z\in\A$ and a complete orbit through $z$. Using the Lyapunov function and some estimates on the alpha limit set of the complete orbit, the authors get the result for the disease-free equilibrium when $\RR_0\leq 1$. For the endemic equilibrium, the authors combine the fact that the derivative of the Lyapunov function must vanish on the alpha-limit set of $z$ (see \cite{SmithThieme2011}) which actually reduces to the endemic equilibrium (by computing the largest invariant set on which the derivative vanishes), with the LAS of the endemic equilibrium to obtain the result when $\RR_0>1$. In \cite{Perasso2019}, since the Lyapunov is computed on the omega-limit set $\omega(x)$ for the initial condition $x$, the LaSalle invariance principle does not imply that the derivative of the Lyapunov function vanishes for each point of $\omega(x)$ as claimed but only on the omega-limit (and alpha-limit) sets of each point of $\omega(x)$ (see \cite{SmithThieme2011}). It follows that the endemic equilibrium belongs to $\omega(x)$ but is not necessarily equal to it. The mistake lies in asserting that, if $y \in \omega(x)$, then $\omega(y) = \omega(x)$. A simple counter-example exists with ODEs (see \cite[Example, p.224]{Teschl2012})) where $\omega(x)$ consists of two fixed points plus the orbits joining them, while $\omega(y)$ reduces to one of the equilibria, for each $y\in\omega(x)$.

In the present paper we do not need the LAS of the endemic equilibrium. The argument follows \cite{MagCluskWebb2010,MartchevaLi2013} by considering $z\in \A$ and a complete orbit through $z$ on which the Lyapunov function is well defined. The LaSalle invariance principle implies that the derivative of the Lyapunov function vanishes both on the omega and alpha limit sets of $z$, which then both reduce to the endemic equilibrium. Using the fact that the Lyapunov is zero at the endemic equilibrium, it follows that the Lyapunov function vanishes along the complete orbit, hence the derivative vanishes for all time and the result follows. This point is of particular interest here, since the stability of the endemic equilibrium cannot be established by linearization when the maximal
$\RR_0$ is attained for multiple populations.

In the second section, we specify the assumptions and notations, remind some definitions then state the main result that is a complete description of the asymptotic behavior for $n$ species in the case where the maximal reproduction number is not necessarily unique, with global asymptotic stability results of the equilibria, extending the results obtained in \cite{MartchevaLi2013}. In the third section we prove the result by induction on $n$, using the tools described above.

\section{Assumptions and results}

For every $k\in\llbracket 1,n\rrbracket$, we define the function
$$\pi_k:a\longmapsto e^{-\int_0^a \mu_k(s)ds}$$
describing the survival probability from infection to the infection age $a$ for the population $k$. We also define the following quantities
$$\overline{\beta_k}:=\sup(\supp(\beta_k)), \qquad r_k:=\int_0^\infty \beta_k(a)\pi_k(a)da$$
which allow us to define the basic reproduction number associated to the species $k$ as
$$\RR_{0,k}=\dfrac{\Lambda r_k}{\mu_S}.$$
We will look at solutions on the
Banach space
$$X:=\R\times (L^1(0,\infty))^n$$
which is endowed with the usual norm while we denote by $X_+$ its positive cone. We also define the sets
$$\S_k:=\left\{(S_0,x_{1,0},...,x_{n,0})\in X_+: \int_0^{\overline{\beta_k}}x_{k,0}(s)ds>0\right\}, \quad \partial \S_k=X_+\setminus \S_k, \quad \forall k\in\llbracket 1,n\rrbracket.$$
\begin{assumption}\mbox{}
\label{Assum:1}
\begin{enumerate}
    \item The constants $\Lambda$ and $\mu_S$ are positive. For each $k\in\llbracket 1,n\rrbracket$ the functions $\mu_k$ and $\beta_k$ belong to $L^\infty(0,\infty)$ with $\beta_k \not\equiv 0$. Moreover, there exists $\mu_0>0$ such that $\mu_S\geq \mu_0$ and such that for each $k\in\llbracket 1,n\rrbracket$ we have
    $$\mu_k(a)\geq \mu_0 \quad \text{ a.e } \quad a\geq 0.$$
    \item For each $k\in\llbracket 1,n\rrbracket$ there exists $\underline{\beta_k}\in[0,\overline{\beta_k})$ such that 
    $$\beta_k(a)>0 \quad \text{a.e.} \quad a\in[\underline{\beta_k},\overline{\beta_k}).$$
    \item The functions $\beta_k$ are uniformly continuous.
\end{enumerate}
\end{assumption}

We note that the last assumption is to obtain the compactness of the orbits. We remind (see \cite{MagCluskWebb2010,Richard2020}) that we can define the linear operator $A:D(A)\subset \X\to \X$ by
$$A\begin{pmatrix}
S \\
\begin{pmatrix}
0 \\
x_1
\end{pmatrix} \\
\vdots \\
\begin{pmatrix}
0 \\
x_n
\end{pmatrix}
\end{pmatrix}=\begin{pmatrix}
-\mu_S S\\
\begin{pmatrix}
-x_1(0)\\
-x_1'-\mu_1 x_1
\end{pmatrix} \\
\vdots \\
\begin{pmatrix}
-x_n(0) \\
-x_n'-\mu_n x_n 
\end{pmatrix}
\end{pmatrix}$$
with $D(A)=\R\times (\{0\}\times W^{1,1}(0,\infty))^n$ and $\X=\R\times (\R\times L^1(0,\infty))^n$. We define the sets 
$\X_0=\overline{D(A)}=\R\times (\{0\}\times L^1(0,\infty))^2$ and its positive cone $\X_{0+}=\R_+\times (\{0\}\times L^1_+(0,\infty))^2$. Then we define the part of $A$ in $\X_0$ as 
$$A_0x=Ax, \ \forall x\in D(A_0):=\{x\in D(A): Ax\in \X_0\}.$$
We also define the non-linear function $F:\X_0\to \X$ by
$$F\begin{pmatrix}
S \\
\begin{pmatrix}
0 \\
x_1
\end{pmatrix} \\
\vdots \\
\begin{pmatrix}
0 \\
x_n 
\end{pmatrix}
\end{pmatrix}=\begin{pmatrix}
\Lambda-S\sum_{k=1}^n\int_0^\infty \beta_k(a)x_k(a)da \vspace{0.1cm} \\
\begin{pmatrix}
S\int_0^\infty \beta_1(a)x_1(a)da \\
0 
\end{pmatrix}\\
\vdots \\
\begin{pmatrix}
S\int_0^\infty \beta_n(a)x_n(a)da \\
0 
\end{pmatrix}
\end{pmatrix}.
$$
Finally, we define the nonlinear generator
$$(A+F)u=Au+F(u), \ \forall u\in D(A)$$
and the part of $A+F$ in $\X_0$ as
\begin{equation}\label{Eq:domain_A+F}
((A+F)_0)u=(A+F)u, \ \forall u\in D((A+F)_0)=\{u\in D(A): Au+F(u)\in \X_0\}.
\end{equation}
We can then rewrite the model \eqref{Eq:model} in the following abstract Cauchy form (not densely defined):
\begin{equation}
    \left\{
    \begin{array}{rcl}
      \dfrac{du}{dt}(t)&=&Au(t)+F(u(t)), \quad \forall t>0  \\
      u(0)&=&u_0\in \X_0. 
    \end{array}
    \right.
    \label{Eq:Cauchy-pb}
\end{equation}
Under Assumption \ref{Assum:1} we can state the following properties (see the proof of \cite[Proposition 2.1]{Richard2020} in the case $n=2$):
\begin{lemma}\mbox{} Suppose that Assumption \ref{Assum:1} holds. Then:
\label{Lemma:prop}
\begin{enumerate}
    \item $A$ is a Hille-Yosida operator with $(-\mu_0,\infty)\subset \rho(A)$, where $\rho(A)$ denotes the resolvent set of $A$ and
    $$\|(\lambda I-A)^{-n})\|_{\L(\X)}\leq \dfrac{1}{(\lambda+\mu_0)^n}, \ \forall \lambda>-\mu_0, \ \forall n\geq 1.$$
    \item $F$ is a locally Lipschitz continuous function: $\forall r>0$, $\exists K_r>0$ such that 
    $$\|F(u_0)-F(\tilde{u}_0)\|_{\X}\leq K_r\|u_0-\tilde{u}_0\|_{\X}$$
    for every $(u_0,\tilde{u}_0)\in (\X_0\cap B_{\X}(0,r))^2$ where $B_{\X}(0,r)$ is the ball centered at $0\in \X$ and radius $r$.
    \item $A$ is resolvent positive and for every $r>0$ there exists $K_r$ such that $F(u_0)+K_ru_0\in \X_+$ and $A-K_rI$ is resolvent positive for every $u_0\in \X_0\cap B_{\X}(0,r)\cap \X_+$.
\end{enumerate}
For the second and third points, the constant $K_r$ can be taken as $K_r:=2r\sum_{k=1}^n \|\beta_k\|_{L^\infty}$.
\end{lemma}
From the above lemma, we deduce (see \cite[Proposition 2.1]{MagCluskWebb2010} in the case $n=1$ or \cite[Theorem 2.2]{Richard2020} for $n=2$) the following result.

\begin{proposition}\label{Prop:solutions}
Suppose that Assumption \ref{Assum:1} holds.
\begin{enumerate}
    \item There exists a unique continuous semiflow $\{U(t)\}_{t\geq 0}$ on $\X_{0+}$ such that for every $u_0\in \X_{0+}$, there exist $t_{\max}(u_0)\leq \infty$ and a continuous map $U(.)u_0\in\Co([0,t_{\max}),\X_{0+})$ which is a maximal integrated solution of \eqref{Eq:Cauchy-pb}, \textit{i.e.}
    $$\int_0^t U(s)u_0ds\in D(A), \quad \forall t\in[0,t_{\max})$$
    and
    $$U(t)u_0=u_0+A\int_0^t U(s)u_0ds+\int_0^t F(U(s)u_0)ds, \quad \forall t\in[0,t_{\max}).$$
    \item This solution is global in time, \textit{i.e.} $t_{\max}(u_0)=+\infty$ and it induces a continuous semiflow via
    $$\Phi:\R_+\times X_+\longmapsto \Phi_t(z)=(S(t),x_1(t,.),...,x_n(t,.))=:(\Phi_t^S,\Phi_t^1,...,\Phi_t^n)(z).$$
    The semiflow rewrites using the following Duhamel formulation
\begin{equation}\label{Eq:Duhamel}
\Phi_t(z)=\left(0,\Phi_t^{1,1}(z),...,\Phi_t^{n,1}(z)\right)+\left(\Phi_t^S(z),\Phi_t^{1,2}(z),...,\Phi_t^{n,2}(z)\right)
\end{equation}
with $\Phi_t^S(z)>0$ for every $t>0$ and every $z\in X_+$. Also for each $k\in\llbracket 1,n\rrbracket$, we have
\begin{equation}\label{Eq:Duhamel_phik}
\left\{
\begin{array}{rcl}
\Phi_t^{k,1}(z)(a)&=&x_{k,0}(a-t)e^{-\int_{a-t}^a \mu_k(s)ds}\chi_{[t,\infty)}(a) \\
\Phi_t^{k,2}(z)(a)&=&\Phi_{t-a}^S(z)\left(\int_0^\infty \beta_k(s)\Phi_{t-a}^k(z)(s)ds\right)e^{-\int_0^a \mu_k(\xi)d\xi}\chi_{[0,t)}(a)
\end{array}
\right.
\end{equation}
where $z=(S_0,x_{1,0},...,x_{n,0})$ and $\chi$ denotes the characteristic function. Also, for every $u_0\in \X_{0+}$ we have
\begin{equation}\label{Eq:estimate_phi_t}
\|U(t)u_0\|_{\X}\leq \dfrac{\Lambda}{\mu_0}\left(1-e^{-\mu_0 t}\right)+\|u_0\|_{\X} e^{-\mu_0 t}\leq \max\left\{\dfrac{\Lambda}{\mu_0},\|u_0\|_{\X}\right\}, \quad \forall t\geq 0.
\end{equation}

\item The semiflow $\Phi$ is bounded-dissipative on $X_+$, \textit{i.e.} there exists a bounded set (that is $B_X(0,\frac{\Lambda}{\mu_0})$, the ball of $X$ centered at $0$ and radius $\frac{\Lambda}{\mu_0}$) which attracts every bounded set of $X_+$.
\item The semiflow $\Phi$ is state-continuous uniformly in finite time, \textit{i.e.} for every $(t,z)\in \R_+\times X_+$ and every $\ep>0$, there exists $\delta>0$ such that $$\|\Phi_s(z)-\Phi_s(\tilde{z})\|_{X}\leq \ep, \qquad \forall (s,\tilde{z})\in[0,t]\times X_+ \ \text{such that } \|z-\tilde{z}\|_{X}\leq \delta.$$
\item The semiflow $\Phi$ is asymptotically smooth, \textit{i.e.} for every non-empty, closed, bounded and positively invariant set $B\subset X_+$, there exists a compact set $K\subset B$ such that $d(\Phi_t(B),K)\underset{t\to \infty}{\to}0$ where we defined
$$d(\Phi_t(B),K)=\sup_{z\in B}\inf_{y\in K}\|\Phi_t(z)-y\|_{X}.$$
\end{enumerate}
\end{proposition}

\begin{proof}
\begin{enumerate}
\item From Lemma \ref{Lemma:prop}, $A$ is a Hille-Yosida operator, so it generates a locally Lipschitz continuous integrated semigroup $\{S_A(t)\}_{t\geq 0}$ on $\X$ (see \cite[Theorem 2.4]{KellermanHieber89} or \cite[Proposition 3.4.3, p.116]{MagalRuan2018}). Moreover, Kellermann-Hieber theorem (see \cite{KellermanHieber89} or \cite[Theorem 3.6.2, p.133]{MagalRuan2018}) implies that for any $\tau>0$ and any $f\in L^1((0,\tau),\X)$, the maps $t\longmapsto \left(S_A \ast f\right)(t)$ are continuously differentiable, where $\ast$ denotes the convolution product:
$$(S_A \ast f)(t)=\int_0^t S_A(t-s)f(s)ds$$
and for all $t\in[0,\tau]$ then
\begin{equation}\label{Eq:convol-estim}
\|u_f(t)\|:=\left\|\dfrac{d}{dt}(S_A \ast f)(t)\right\|\leq e^{-\mu_0 t}\int_0^t e^{\mu_0 s}\|f(s)\| \mathrm{d}s\leq \delta_A(t)\underset{s\in[0,t]}{\sup}||f(s)||.
\end{equation}
where $\delta_A(t)=t$ converges to $0$ as $t\to0$. Then Assumption 5.1.2 from \cite{MagalRuan2018} applies and, since we proved that $F$ is locally Lipschitz continuous, \cite[Theorem 5.2.7, p.226]{MagalRuan2018} states that there exists a unique maximal integrated (mild) solution $U(.)u_0\in\Co([0,t_{\max}),\X_0)$ (with $t_{\max}\leq \infty$) of the problem \eqref{Eq:Cauchy-pb} for each initial condition $u_0\in \X_0$. This solution satisfies:
\begin{equation}\label{Eq:u_convol}
    U(t)u_0=T_{A_0}(t)u_0+\dfrac{d}{dt}\left(S_A \ast (F\circ U(.)u_0)(t)\right).
\end{equation}
where $\{T_{A_0}\}_{t\geq 0}$ is the $\Co_0$-semigroup on $\X_0$ generated by $A_0$. Note that the existence of $\{T_{A_0}\}_{t\geq 0}$ on $\overline{D(A_0)}$ comes from the Hille-Yosida Theorem \cite[Lemma 2.4.5]{MagalRuan2018} while \cite[Lemma 2.2.11]{MagalRuan2018} implies that $\overline{D(A_0)}= \X_0$.

Due to the positive properties stated in Lemma \ref{Lemma:prop}, we deduce by \cite[Proposition 5.3.2, p. 227]{MagalRuan2018} that for each $u_0\in \X_{0+}$ the solution satisfies $U(.)u_0\in \Co([0,t_{\max}),\X_+)$.
This proves the first point.

\item Let $u_0=(S_0,(0,x_{1,0}),...,(0,x_{n,0}))\in \X_{0+}$ and $
z:=(S_0,x_{1,0},...,x_{n,0})\in X_+$. First, the formula \eqref{Eq:Duhamel_phik} are derived for every $t\in[0,t_{\max}(u_0))$ in \cite[Theorem 2.2]{Richard2020} by the method of characteristics. Next, we have $U(.)u_0\in \Co([0,t_{\max}(u_0)),\X_{0+})$.
Denoting $U(t)u_0=(S(t),(0,x_1(t,.)),...,(0,x_n(t,.)))$ for every $t\in[0,t_{\max}(u_0))$, then the function $t\longmapsto S(t)$ belongs to $\Co^1([0,t_{\max}(u_0)),\R_+)$  and the $S$-equation of \eqref{Eq:model} holds. We deduce that
\begin{flalign*}
S(t)&\leq S_0e^{-\mu_0 t}+\int_0^t\left(\Lambda-S(s)\sum_{k=1}^n \int_0^\infty \beta_k(a)x_k(s,a)da\right)e^{-\mu_0 (t-s)}ds \\
&\leq S_0e^{-\mu_0 t}+\dfrac{\Lambda}{\mu_0}\left(1-e^{-\mu_0 t}\right)-\sum_{k=1}^n \int_0^t S(s)\left(\int_0^\infty \beta_k(a)x_k(s,a)da\right)e^{-\mu_0(t-s)}ds
\end{flalign*}
On the other hand, using \eqref{Eq:Duhamel_phik} we have
$$\int_0^\infty x_k(t,a)da\leq \|x_{k,0}\|_{L^1}e^{-\mu_0t}+\int_0^t S(t-a)\left(\int_0^\infty \beta_k(s)x_k(t-a,s)ds\right)e^{-\mu_0 a}da.$$
It follows that \eqref{Eq:estimate_phi_t} holds for every $t\in[0,t_{\max}(u_0))$ whence the solution $U(.)u_0$ is global in time (otherwise the norm should explode at $t_{\max}(u_0)$). We deduce the existence of the continuous semiflow:
$$\Phi:\R_+\times X_+\ni (t,z) \longmapsto (S(t),x_1(t,.),...,x_n(t,.))\in X_+$$
with the bijection
$$z:=(S_0,x_{1,0},...,x_{n,0})\in X_+ \longmapsto u_0=(S_0,(0,x_{1,0}),...,(0,x_{n,0}))\in \X_{0+}.$$
This proves the second point.

\item Let $B\subset X_+$ be a bounded subset and let $z:=(S_0,x_{1,0},...,x_{n,0})\in B$. From \eqref{Eq:estimate_phi_t} we know that
\begin{equation*}
\|\Phi_t(z)\|_{X}\leq \dfrac{\Lambda}{\mu_0}\left(1-e^{-\mu_0 t}\right)+\|z\|_{X} e^{-\mu_0 t}, \quad \forall t\geq 0.
\end{equation*}
It shows that $\lim_{t\to+\infty}\|\Phi_t(z)\|_{X}\leq \frac{\Lambda}{\mu_0}$ uniformly in $z\in B$. In conclusion we get
$$d\left(\Phi_t(B),B_X\left(0,\frac{\Lambda}{\mu_0}\right)\right)\underset{t\to \infty}{\longrightarrow}0$$
and $\Phi$ is bounded-dissipative on $X_+$.

\item Let $u_0\in \X_{0+}$. Then $U(.)u_0\in \Co(\R_+,\X_{0+})$. Let $t\geq 0$, $\ep>0$ and define
$$r=\max\left\{\dfrac{\Lambda}{\mu_0},\|u_0\|_{\X}+\ep\right\}, \qquad \delta=\ep e^{-K_r t}$$
where $K_r$ is as in Lemma \ref{Lemma:prop}. Let $\tilde{u}_0\in \X_{0+}$ such that $\|u_0-\tilde{u}_0\|_{\X}\leq \delta$. 
We have $U(.)\tilde{u}_0\in \Co(\R_+,\X_{0+})$ and $F$ is Lipschitz continuous whence $F\circ U(.)u_0$ and $F\circ U(.)\tilde{u}_0$ belong to $L^1([0,t],\X)$. Also we have $(U(s)u_0,U(s)\tilde{u}_0)\in (\X_{0+}\cap B_{\X}(0,r))^2$ for every $s\in[0,t]$ by using \eqref{Eq:estimate_phi_t} and since $\|\tilde{u}_0\|_{\X}\leq \delta+\|u_0\|_{\X}\leq r$. From \eqref{Eq:u_convol} we deduce that for every $s\in[0,t)$:
$$U(s)\tilde{u}_0-U(s)u_0=T_{A_0}(s)(\tilde{u}_0-u_0)+\dfrac{d}{dt}\left(S_A \ast ((F\circ U(.)\tilde{u}_0)-(F\circ U(.)u_0))(s)\right).$$
From \cite[Lemma 2.2.10, p. 65]{MagalRuan2018}, we have $\rho(A_0)=\rho(A)$ whence Hille-Yosida theorem states that
$$\|T_{A_0}(s)(\tilde{u}_0-u_0)\|_{\X}\leq e^{-\mu_0 s}\|\tilde{u}_0-u_0\|_{\X}\leq \delta, \ \forall s\in[0,t].$$
Using again Kellermann-Hieber theorem then \eqref{Eq:convol-estim} holds and we deduce that
$$\|U(s)\tilde{u}_0-U(s)u_0\|_{\X}\leq \delta+\int_0^s \|F(U(\xi)\tilde{u}_0)-F(U(\xi)u_0)\|_{\X}d\xi, \ \forall s\in[0,t].$$
It follows from Lemma \ref{Lemma:prop} that
$$\|U(s)\tilde{u}_0-U(s)u_0\|_{\X}\leq \delta+K_r\int_0^s \|U(\xi)\tilde{u}_0-U(\xi)u_0\|_{\X}d\xi, \quad \forall s\in[0,t]$$
which leads by means of Gronwall's inequality to
\begin{equation*}
\|U(s)\tilde{u}_0-U(s)u_0\|_{\X}\leq \delta e^{K_r s} \leq \ep, \quad \forall s\in [0,t]
\end{equation*}
which proves the point.

\item From the decomposition \eqref{Eq:Duhamel}, we see that the function $\eta:\R_+\times \R_+\ni (t,r)\longmapsto  nre^{-\mu_0 t}\in\R_+$ satisfies:
$$\forall r>0: \lim_{t\to +\infty}\eta(t,r)=0$$
and
$$\|(0,\Phi_t^{1,1},...,\Phi_t^{n,1})\|_{X}\leq \eta(t,r), \ \forall (t,z)\in \R_+\times X_+ \ \text{such that } \|z\|_{X}\leq r.$$ Let $t>0$ and $k\in\llbracket 1,n\rrbracket$. We want to show that the operator $\Phi_t^{k,2}:X_+\to L^1(\R_+)$ is compact. Let $\B\subset X_+$ be a bounded subset and $r=\sup_{z\in \B}\|z\|_{X}$. We use Riez-Fréchet-Kolmogorov theorem  \cite[Theorem X.1, p. 275]{Yosida80} to prove that $\{\Phi_t^{k,2}(z), z\in \B\}$ is relatively compact in $L^1$. First, using \eqref{Eq:estimate_phi_t} we get
$$\sup_{z\in \B}\|\Phi_t^{k,2}(z)\|_{L^1}\leq \left(\max\left\{\dfrac{\Lambda}{\mu_0},r\right\}\right)<+\infty.$$
Secondly, let $h>0$ then clearly
$$\sup_{z\in \B}\int_h^\infty \Phi_t^{k,2}(z)(a)da\underset{h\to +\infty}{\longrightarrow} 0.$$
since $\Phi_t^{k,2}(z)(a)=0$ for $a>t$. Thirdly, let $h>0$ such that $t-h>0$. Let us define
$$B^{z}_k(s)=\Phi_s^k(z)(0)=\Phi^S_s(z)\int_0^\infty \beta_k(a)\Phi_s^k(z)(a)da, \quad \forall (z,s)\in \B\times \R_+$$
whence
$$\Phi_s^{k,2}(z)(a)=B^z_k(s-a)\pi_k(a)\chi_{[0,s)}(a), \quad \forall (z,s,a)\in \B\times[0,t]\times \R_+.$$
Using \eqref{Eq:estimate_phi_t} we have
$$B^z_k(s)\leq \|\beta_k\|_{L^\infty}\left( \max\left\{\dfrac{\Lambda}{\mu_0}, r\right\}\right)^2=:c_1, \quad \forall (z,s)\in \B\times \R_+$$
and with the $S$-equation of \eqref{Eq:model} we get
$$\left|\dfrac{d\Phi_s^S(z)}{ds}\right|\leq \Lambda+\mu_S \max\left\{\dfrac{\Lambda}{\mu_0}, r\right\}+ \left(\max\left\{\dfrac{\Lambda}{\mu_0}, r\right\}\right)^2\sum_{j=1}^n \|\beta_j\|_{L^\infty}=:c_2, \quad \forall (z,s)\in \B\times \R_+$$
leading to
$$|\Phi_{s+h}^S(z)-\Phi_{s}^S(z)|\leq c_2 h, \quad \forall (z,s)\in \B\times \R_+.$$
We remark that
$$B^z_k(s)=\Phi_s^S(z)\int_0^s \beta_k(a)B^z_k(s-a)e^{-\int_0^a \mu(\xi)d\xi}da+\Phi^S_{s}(z)\int_0^\infty \beta_k(a+s)\Phi_0^k(z)(a)e^{-\int_a^{a+s}\mu_k(\xi)d\xi}da
$$
for every $(z,s)\in\B\times \R_+$ whence
$$|B^z_k(s+h)-B^z_k(s)|\leq c_3(h)+\|\beta_k\|_{L^\infty}\max\left\{\dfrac{\Lambda}{\mu_0},r\right\}\int_0^s |B^z_k(s+h-a)-B^z_k(s-a)|da, \ \forall (z,s)\in \B\times \R_+$$
where
\begin{flalign*}
c_3(h):=&r\max\left\{\dfrac{\Lambda}{\mu_0},r\right\}\left(c_1 h\|\beta_k\|_{L^\infty}+\sup_{a\geq 0}|\beta_k(a+h)-\beta_k(a)|+(1-e^{-h\|\mu_k\|_{L^\infty}})\|\beta_k\|_{L^\infty}\right)\\
&+c_2h\|\beta_k\|_{L^\infty} \left(\dfrac{c_1}{\mu_0}+r\right)
\end{flalign*}
with $c_3(h)\underset{h\to 0}{\longrightarrow} 0$ uniformly in $z\in \B$ since $\beta_k$ is uniformly continuous by Assumption \ref{Assum:1}. From Gronwall's inequality we deduce that
$$|B^z_k(s+h)-B^z_k(s)|\leq c_3(h)\exp\left(r(t+h) \times\max\left\{\frac{\Lambda}{\mu_0},r\right\}\right)=c_4(h), \quad \forall (z,s)\in \B\times [0,t].$$
Thus it comes:
\begin{equation*}
    \begin{array}{rcl}
    &&\displaystyle\int_0^\infty |\Phi_t^{k,2}(z)(a+h)-\Phi_t^{k,2}(z)(a)|da\leq  \int_{-h}^{t-h} |B_k^z(t-(a+h))-B_k^z(t-a))|\pi_k(a+h)da \\
    &&\qquad +\displaystyle\int_{-h}^{t-h} B_k^z(t-a)|\pi_k(a+h)-\pi_k(a)|da+\int_{-h}^0 B_k^z(t-a)\pi_k(a)da+\int_{t-h}^t B_k^z(t-a)\pi_k(a)da 
    \end{array}
\end{equation*}
whence
$$\int_0^\infty |\Phi_t^{k,2}(z)(a+h)-\Phi_t^{k,2}(z)(a)|da\leq tc_4(h)+tc_1(1-e^{-h\|\mu_k\|_{L^\infty}})+2hc_1\underset{h\to 0}{\longrightarrow} 0
$$
uniformly in $z\in \B$. It follows that the operator $\Phi_t^{k,2}$ is compact for each $t>0$. In conclusion the operator $(\Phi_t^S, \Phi_t^{1,2}, ..., \Phi_t^{n,2})$ is compact since the range of $\Phi_t^S$ is finite dimensional. It proves the relative compactness of the positive orbits by using \cite[Proposition 3.13, p. 100]{Webb85}, for each $z\in X_+$ and also that the semiflow $\{\Phi_t\}_{t\geq 0}$ is asymptotically smooth (see \cite[Lemma 3.2.6, p. 38]{Hale1988} or \cite[Theorem 2.46, p.51]{SmithThieme2011}).

\end{enumerate}
\end{proof}

\begin{corollary}\label{Cor:attractor} Suppose that Assumption \ref{Assum:1} holds.There exists a strong global attractor $\A\subset X_+$, that is $\A$ is a non-empty, compact and invariant subset that attracts every bounded subsets of $X_+$. This attractor satisfies $\A\subset B_{X}(0,\frac{\Lambda}{\mu_0})$ which is the ball centered at $0\in X$ with radius $\frac{\Lambda}{\mu_0}$. The set $\A$ is locally asymptotically stable (LAS), \textit{i.e.} for every neighborhood $\V$ of $\A$, there exists a neighborhood $\W\subset \V$ of $\A$ such that $\Phi_t(\W)\subset \V$ for every $t\geq 0$ and  there exists moreover a neighborhood $\V$ of $\A$ such that $\A$ attracts $\V$. Finally the following holds:
\begin{equation}\label{Eq:c_s}
\exists \ c_S>0: \forall (S_0,x_{1,0},...,x_{n,0})\in \A, \ \dfrac{\Lambda}{\mu_0}\geq S_0\geq c_S.
\end{equation}
The constant $c_S$ can be taken as $c_S:=\dfrac{\Lambda}{\mu_S+\frac{\Lambda}{\mu_0}\sum_{k=1}^n \|\beta_k\|_{L^\infty}}$.
\end{corollary}

\begin{proof}
First we see that orbits of bounded sets are bounded since
$$\|\Phi_t(z)\|_{X}\leq \max\left\{\|z\|_{X}, \dfrac{\Lambda}{\mu_0}\right\}, \quad \forall t\geq 0$$
by using \eqref{Eq:estimate_phi_t}. Using \cite[Theorem 2.33, p. 43]{SmithThieme2011} (see also \cite[Theorem 3.4.6, p. 39]{Hale1988}), we deduce that there exists a strong global attractor $\A\subset X_+$. This attractor $\A$ is locally asymptotically stable according to \cite[Theorem 3.4.2, p. 39]{Hale1988} or \cite[Theorem 2.39 p. 47]{SmithThieme2011}. Moreover, since $\Phi$ is bounded dissipative by Proposition \ref{Prop:solutions} and since $\A$ is compact and invariant then for every $t\geq 0$:
$$d\left(\A,B_{X}\left(0,\frac{\Lambda}{\mu_0}\right)\right)=d\left(\Phi_t(\A),B_{X}\left(0,\frac{\Lambda}{\mu_0}\right)\right)\underset{t\to \infty}{\longrightarrow}0$$
whence $\A\subset B_{X}(0,\frac{\Lambda}{\mu_0})$.

For the last point, let $z:=(S_0,x_{1,0},...,x_{n,0})\in \A$. By invariance of $\A$ there exists a complete orbit $\{\phi(t), t\in \R\}\subset \A$ through $z$, that satisfies $\|\phi(t)\|_{X}\leq \frac{\Lambda}{\mu_0}$ for every $t\in \R$. From \eqref{Eq:model} we get
$$\dfrac{d\Phi_t^S(\phi(s))}{dt}\geq \Lambda-\left(\mu_S+\dfrac{\Lambda}{\mu_0}\sum_{k=1}^n \|\beta_k\|_{L^\infty}\right) \Phi_t^S(\phi(s)), \quad \forall (t,s)\in \R_+^*\times \R$$
leading for every $(t,s)\in\R_+\times \R$ to
$$\Phi_t^S(\phi(s))\geq \dfrac{\Lambda}{\mu_S+\frac{\Lambda}{\mu_0}\sum_{k=1}^n \|\beta_k\|_{L^\infty}}\left(1-e^{-(\mu_S+\frac{\Lambda}{\mu_0}\sum_{k=1}^n \|\beta_k\|_{L^\infty})t}\right)\underset{t\to+\infty}{\longrightarrow} \dfrac{\Lambda}{\mu_S+\frac{\Lambda}{\mu_0}\sum_{k=1}^n \|\beta_k\|_{L^\infty}}=c_S.$$
Let $\ep>0$ then there exists $t^*\geq 0$ such that for every $t\geq t^*$ we have $\Phi_t^S(\phi(s))\geq c_S-\ep$ for every $s\in \R$. Taking $s=-t$ leads to $S_0\geq c_S-\ep$ for any $\ep>0$, whence \eqref{Eq:c_s} holds.
\end{proof}

As it is proved in \cite[Proposition 4.4.4]{Richard2020} in the case $n=2$, we know that for each $k\in\llbracket 1,n\rrbracket$, the sets $\S_k$ and $\partial \S_k$ are positively invariant. Also, we see that if $z\in \partial \S_k$ then 
\begin{equation}\label{Eq:estimate_phik_t}
\|\Phi_t^k(z)\|_{L^1(\R_+)}\leq \|x_{0,k}\|_{L^1(\R_+)}e^{-\mu_0 t}, \qquad \|\beta_k x_k(t,.)\|_{L^1}= 0 \qquad \forall t\geq 0
\end{equation}
which reduces the problem to the $n-1$ populations model. Without loss of generality, we may assume that
\begin{equation}\label{Hyp:R0-decrease}
\RR_{0,1}\geq \RR_{0,2}\geq \cdots \geq \RR_{0,n}.
\end{equation}
We define now the following indexing:
$$\sigma_0=0, \qquad \sigma_k=\max\{j\in\llbracket 1,n\rrbracket : \RR_{0,j}=\RR_{0,1+\sigma_{k-1}}\} \ \forall k\in\llbracket 1,n_{\RR}\rrbracket$$
where $n_{\RR}\in \llbracket 1,n\rrbracket$ is the number of different reproduction numbers and satisfies $\sigma_{n_{\RR}}=n$. For example if $n=1$ or if every reproduction number are equal then $n_{\RR}=1$ ; while in the case where each reproduction number is different then $n_{\RR}=n$. We denote by $n_{>}\in\llbracket 0, n_{\RR}\rrbracket$ the number of different reproduction numbers that are bigger than one, which is such that 
$$\RR_{0,k}>1, \ \forall k\in\llbracket 1,\sigma_{n_{>}}\rrbracket.$$
With these notations, we can describe the different equilibria. First there is the disease-free equilibrium:
$$E_0=\left(S_0^*,0_{(L^1(\R_+))^n}\right)=\left(\dfrac{\Lambda}{\mu_S},0,...,0\right)\in X_+$$
which always exist. Second, if $n_>\geq1$ then for each $k\in\llbracket 1,n_{>}\rrbracket$ the following set of equilibria exists:
\begin{equation}\label{Eq:equilibria-k}
 \E_{k}=\left\{E^{*,k}_{\alpha_1,...,\alpha_n}: (\alpha_1,...,\alpha_n)\in[0,1]^n, \hspace{-0.2cm} \sum_{j=1+\sigma_{k-1}}^{\sigma_k}\hspace{-0.2cm}\alpha_j=1, \alpha_j=0 \ \forall j\in\llbracket 1,\sigma_{k-1}\rrbracket\cup\llbracket 1+\sigma_k,n\rrbracket \right\}
\end{equation}
where
\begin{equation*}
    \left\{
    \begin{array}{rl}
    E^{*,k}_{\alpha_1,...,\alpha_n}&=(S^*_{\sigma_k},x^*_{1,\alpha_1},\cdots, x^*_{n,\alpha_n}),\\
         S^*_{\sigma_k}&=\dfrac{1}{r_{\sigma_k}},\\
         x^*_{j,\alpha_j}(a)&=\dfrac{\mu_S(\RR_{0,j}-1)}{r_j}\alpha_j\pi_j(a), \quad \forall (j,a)\in\llbracket 1,n\rrbracket\times \R_+.
    \end{array}
    \right.
\end{equation*}
We can note that if for some $k$ we have $\sigma_k=1+\sigma_{k-1}$ (for example if $n=1$) then $\E_k$ is reduced to a single equilibrium (otherwise there are a infinite number of equilibria). Following the classical notations (see  \cite{Hale1988}), we denote for $z\in X$:
$$\gamma^+(z)=\left\{\Phi_t(z)), \ t\geq 0\right\}, \qquad \omega(z)=\bigcap_{\tau \geq 0}\overline{\{\Phi_t(z), \ t\geq \tau\}}$$
respectively the positive orbit starting from $z$ and the $\omega$-limit set of $z$. Also, let a function $\phi:(-\infty,0]\to X$ such that $\phi(0)=z$ and for any $s\leq 0$, $\Phi_t(\phi(s))=\phi(t+s)$ for each $0\leq t\leq -s$, then the set $\{\phi(s), s\leq 0\}$ is called a negative orbit through $z$, denoted by $\gamma^-(z)$. Similarly, for a function $\phi:\R\to X$ such that $\phi(0)=z$ and for each $s\in \R$, $\Phi_t(\phi(s))=\phi(t+s)$ for every $t\geq 0$, then the set $\{\phi(s), s\in \R\}$ is called a complete orbit through $z$, denoted by $\gamma(z)$. We define the set
$$H(t,z)=\{y\in X: \text{there is a negative orbit through } z \text{ defined by } \phi:(-\infty,0]\to X, \phi(0)=z \text{ and } \phi(-t)=y\}$$
and the $\alpha$-limit set of $z$ as
$$\alpha(z)=\bigcap_{\tau \geq 0}\overline{\{H(t,z), \ t\geq \tau\}}.$$
For a given complete orbit $\gamma(z)=\{\phi(s),s\in \R\}$, we define the $\alpha$-limit set of the orbit similarly as
$$\alpha(\phi)=\alpha(\gamma(z))=\bigcap_{\tau\geq 0}\overline{\{\phi(-t), t\geq \tau\}}.$$
Furthermore, we define the sets
$$\Gamma^-(z)=\bigcup_{t\geq 0}H(t,z), \qquad \Gamma(z)=\Gamma^-(z)\cup \gamma^+(z)$$
that contain respectively all negative and complete orbits through $z$. To obtain results about the asymptotic behavior of the solutions, we will make use of Lyapunov functionals. To this end we define the following key function
$$g:\R_+^*\ni x \longmapsto x-\ln(x)-1\in \R_+$$
and for each $k\in\llbracket 1,n\rrbracket$ we define the function $\Psi_k\in L^\infty_+(0,\infty)$ by
$$\Psi_k(a)=\dfrac{1}{r_k}\int_a^\infty \beta_k(s)e^{-\int_a^s \mu_k(\xi)d\xi}ds.$$
We now define for every $z:=(S_0,x_{1,0},...,x_{n,0})$ the functional
\begin{equation}\label{Eq:L0}
L_0:z\longmapsto S_0^*g\left(\dfrac{S_0}{S_0^*}\right)+\sum_{k=1}^n \int_0^\infty \Psi_k(a)x_{k,0}(a)da
\end{equation}
and for each $k\in\llbracket 1,n_{\RR}\rrbracket$ we define the functional
\begin{equation}\label{Eq:Lk}
L_k^{\alpha_{1+\sigma_{k-1}},...,\alpha_{\sigma_k}}:z\longmapsto S_{\sigma_k}^*g\left(\dfrac{S_0}{S_{\sigma_k}^*}\right)+\sum_{j=1}^n \int_0^\infty \eta_j\Psi_j(a)x_{j,0}(a)da+\sum_{j=1}^n\int_0^\infty (1-\eta_j) \Psi_j(a) x^*_{j,\alpha_j}(a)g\left(\dfrac{x_{j,0}(a)}{x^*_{j,\alpha_j}(a)}\right)da
\end{equation}
where
$$\eta_j=\begin{cases}
1 & \text{ if } j\in\llbracket 1, \sigma_{k-1}\rrbracket\cup \llbracket 1+\sigma_k,n\rrbracket \ \text{ or if } \alpha_j=0 \\
0 & \text{else.}
\end{cases}$$

Under these assumptions, we state the following result which is mainly a competitive exclusion principle but where multiple populations can survive provided that they have the same $\RR_0$ and provided that the populations with bigger $\RR_{0,k}$ are not present in the system (\textit{i.e.} belong to the corresponding set $\partial \S_k$).

\begin{theorem} \label{Thm:GAS}
Suppose that \eqref{Hyp:R0-decrease} and Assumption \ref{Assum:1} hold.
\begin{enumerate}
\item The disease-free equilibrium $E_0$ is GAS on $X_+$ if $n_>=0$ and in $(\cap_{j=1}^{\sigma_{n_>}} \partial \S_j)$ otherwise. Moreover, let $K\subset (\cap_{j=1}^{\sigma_{n_>}}\partial \S_j)$ be a compact invariant, then the functional $L_0$ defined by \eqref{Eq:L0} is a Lyapunov functional on every complete orbit $\gamma(z)$ for any $z\in K$.
\item For every $k\in\llbracket 1, n_>\rrbracket$, the set $\E_k$ is GAS in $(\cap_{j=1}^{\sigma_{k-1}} \partial \S_j)\cap (\cup_{j=1+\sigma_{k-1}}^{\sigma_k} \S_j)$. More precisely: for each non-empty subset $\J\subset \llbracket 1+\sigma_{k-1},\sigma_k\rrbracket$, the set of equilibria 
\begin{equation}\label{Eq:equilibria_kJ}
\E_{k,\J}:=\{E^{*,k}_{\alpha_1,...,\alpha_n}=(S^*_{\sigma_k},x^*_{1,\alpha_1},...,x^*_{n,\alpha_n})\in \E_k: \alpha_j=0 \ \forall j\in\llbracket 1,n\rrbracket \setminus \J\}
\end{equation}
is GAS in $(\cap_{j=1}^{\sigma_{k-1}}\partial \S_j)\cap (\cup_{j\in \J} \S_j)\cap(\cap_{j\in \llbracket 1+\sigma_{k-1},\sigma_k\rrbracket\setminus \J} \partial \S_j)$. Also, let $K$ be a compact invariant such that $K\subset (\cap_{j=1}^{\sigma_{k-1}}\partial \S_j)\cap (\cap_{j \in \J}\S_j)\cap (\cap_{j\in\llbracket 1+\sigma_{k-1},\sigma_k\rrbracket \setminus \J}\partial \S_j)$ then the functional $L_k^{\alpha_{1+\sigma_{k-1}},...,\alpha_{\sigma_k}}$ defined by \eqref{Eq:Lk}, where $(\alpha_1,...,\alpha_n)$ is such that $E^{*,k}_{\alpha_1,...,\alpha_n}\in \E_{k,\J}$, is a Lyapunov functional on every complete orbit $\gamma(z)$ for any $z\in K$.
\end{enumerate}
\end{theorem}

\section{Proof of the global stability}

In this section we prove Theorem \ref{Thm:GAS} by induction on $n\geq 1$. The initialisation is clear since it is handled in \cite{MagCluskWebb2010}.

\noindent Let $n\geq 2$ and we suppose that the statements of Theorem \ref{Thm:GAS} are true for $n-1$ (in others words, for the system \eqref{Eq:model} with only $n-1$ infected populations then the asymptotic behavior is known for every case). We will prove that the statements are still true for $n$.

If $z\in \partial \S_1$, then $\Phi_t(z)\in \partial \S_1$ by invariance of $\partial \S_1$ and from \eqref{Eq:estimate_phik_t} it is clear that $\|x_{1}(t,.)\|_{L^1(\R_+)}\to 0$ when $t\to \infty$. Consequently the model \eqref{Eq:model} is reduced to $n-1$ populations (which are $\llbracket 2,n\rrbracket$) and by recurrence hypothesis we get the first item when $n_>\neq 0$ and the second item is also true for every $k\in\llbracket 2,n_>\rrbracket$.

In all that follow we assume that \eqref{Hyp:R0-decrease} and Assumption \ref{Assum:1} hold. We start with the following lemmas.

\begin{lemma}\label{Lemma:L0}
Let $v\in \A$, then the function $t\longmapsto L_0(\Phi_t(v))$ is well-defined on $\R_+$ and for every $t>0$ it satisfies
\begin{equation}
\begin{array}{rll}
    \label{Eq:L0_deriv}
   \dfrac{dL_0(\Phi_{t}(v))}{dt}&=-\dfrac{(\Lambda-\mu_S \Phi_t^S(v))^2}{\mu_S \Phi_t^S(v)}+\displaystyle \sum_{k=1}^n \left(\dfrac{\RR_{0,k}-1}{r_k}\right)\int_0^\infty \beta_k(a)\Phi_t^k(v)(a)da.
\end{array}
\end{equation}
In particular $L_0$ is a Lyapunov functional on $\A \cap \left(\bigcap_{j=1}^{\sigma_{n_>}}\partial \S_j\right)$.
\end{lemma}

\begin{proof}
Using \eqref{Eq:c_s} implies that $L_0(\Phi_t(v))$ is well-defined on $\R_+$. Let $v:=(S_0,x_{1,0},...,x_{n,0})\in \A$ and \break $\widehat{v}:=(S_0,(0,x_{1,0}),...,(0,x_{n,0}))\in \X_+$. If $\widehat{v}\in D((A+F)_0)\cap \X_{0+}$ (we remind the definition \eqref{Eq:domain_A+F}) then using \cite[Theorem 5.6.6, p. 242]{MagalRuan2018} we may compute the derivative of the function $t\longmapsto L_0(\Phi_t(v))$ according to $t$ since $\Phi_t^k(v)$ belongs to $W^{1,1}(0,\infty)$ for each $t\geq 0$ and every $k\in\llbracket 1,n\rrbracket$. Similar computations to \cite[Proposition 5.3.1]{Richard2020} lead for every $t>0$ to \eqref{Eq:L0_deriv} which can be rewritten in the following form for every $t\geq 0$
\begin{equation}\label{Eq:L0_int}
\begin{array}{rll}
L_0(\Phi_{t}(v))&=L_0(v)-\displaystyle\int_0^t \dfrac{(\Lambda-\mu_S\Phi_{\xi}^S(v))^2}{\mu_S\Phi_{\xi}^S(v)}d\xi +\sum_{k=1}^n \left(\dfrac{\RR_{0,k}-1}{r_k}\right)\int_0^t \int_0^\infty \beta_k(a)\Phi_{\xi}^k(v)(a)da d\xi.
\end{array}
\end{equation}
If $\widehat{v}\not\in D((A+F)_0)\cap \X_{0+}$ then we let $t\geq 0$ and we use the density of $D((A+F)_0)\cap \X_{0+}$ into $\X_{0+}$ (see \cite[Lemma 5.6.7, p. 243]{MagalRuan2018}) to compute the derivative of the function $L_0$ along the solution of \eqref{Eq:model} for the initial condition $v$. Thus we can assert that
$$\exists \ (\widehat{v}^{(m)})_{m\in\N}\subset (D((A+F)_0)\cap \X_{0+})^{\N}:  \|\widehat{v}^{(m)}-\widehat{v}\|_{\X}\underset{m\to \infty}{\longrightarrow} 0$$
where $\widehat{v}^{(m)}=(S_{0}^{(m)},(0,x_{1,0}^{(m)}),...,(0,x_{n,0}^{(m)}))\in D((A+F)_0)\cap \X_{0+}$. Denoting $v^{(m)}=(S_{0}^{(m)},x_{1,0}^{(m)},...,x_{n,0}^{(m)})\in X_{+}$ and using Proposition \ref{Prop:solutions} (4), we see that this sequence satisfies
$$\sup_{s\in[0,t]}\|\Phi_s(v^{(m)})-\Phi_s(v)\|_{X}\underset{m\to \infty}{\longrightarrow}0.$$
Let $\ep>0$ be small enough and $m_{\ep}\in \N$ such that
\begin{equation}\label{Eq:estimates_vm}
\|\widehat{v}^{(m)}-\widehat{v}\|_{\X}\leq \ep, \quad \sup_{s\in[0,t]}\|\Phi_s(v^{(m)})-\Phi_s(v)\|_{X}\leq \ep, \quad \forall m\geq m_{\ep}.
\end{equation}
Using \eqref{Eq:c_s} we get $S_0^{(m)}\geq c_S-\ep>0$ for each $m\geq m_{\ep}$ so it is clear that the function $s\longmapsto L_0(\Phi_s(v^{(m)}))$ is well-defined on $[0,t]$ and its derivative can be computed as for \eqref{Eq:L0_deriv} leading to
\begin{equation*}
\begin{array}{rll}
L_0(\Phi_{s}(v^{(m)}))&=L_0(v^{(m)})-\displaystyle\int_0^s \dfrac{(\Lambda-\mu_S\Phi_{\xi}^S(v^{(m)}))^2}{\mu_S\Phi_{\xi}^S(v^{(m)})}d\xi +\sum_{k=1}^n \left(\dfrac{\RR_{0,k}-1}{r_k}\right)\int_0^s \int_0^\infty \beta_k(a)\Phi_{\xi}^k(v^{(m)})(a)da d\xi
\end{array}
\end{equation*}
for every $s\in[0,t]$ and every $m\geq m_{\ep}$. The goal is to show that
\begin{equation}\label{Eq:L0_v}
\begin{array}{rll}
L_0(\Phi_{s}(v))&=L_0(v)-\displaystyle\int_0^s \dfrac{(\Lambda-\mu_S\Phi_{\xi}^S(v))^2}{\mu_S\Phi_{\xi}^S(v)}d\xi +\sum_{k=1}^n \left(\dfrac{\RR_{0,k}-1}{r_k}\right)\int_0^s \int_0^\infty \beta_k(a)\Phi_{\xi}^k(v)(a)da d\xi.
\end{array}
\end{equation}
For this, we compute for every $s\in[0,t]$
\begin{equation*}
    \begin{array}{rcl}
    &&\left|L_0(\Phi_s(v))-L_0(v)+\displaystyle\int_0^s \dfrac{(\Lambda-\mu_S\Phi_{\xi}^S(v^{(m)}))^2}{\mu_S \Phi_{\xi}^S(v^{(m)})}d\xi -\sum_{k=1}^n \left(\dfrac{\RR_{0,k}-1}{r_k}\right)\int_0^s \int_0^\infty \beta_k(a)\Phi_{\xi}^k(v^{(m)})(a)dad\xi  \right| \vspace{0.2cm}\\ 
    & \leq& \left|L_0(\Phi_s(v))-L_0(\Phi_s(v^{(m)}))\right| +s\times \sup_{\xi\in[0,s]}\left|\dfrac{(\Lambda-\mu_S\Phi_{\xi}^S(v))^2}{\mu_S \Phi_{\xi}^S(v)}-\dfrac{(\Lambda-\mu_S\Phi_{\xi}^S(v^{(m)}))^2}{\mu_S \Phi_{\xi}^S(v^{(m)})}\right| \\
    & &+\left|L_0(v)-L_0(v^{(m)})\right|+\displaystyle\sum_{k=1}^n \left(\dfrac{\RR_{0,k}-1}{r_k}\right) s\times \|\beta_k\|_{L^\infty} \sup_{\xi\in[0,s]}\|\Phi_{\xi}^k(v)-\Phi_{\xi}^k(v^{(m)})\|
    \end{array}
\end{equation*}
Since $\A$ is invariant then by using \eqref{Eq:c_s} we get $\Phi_s^S(v)\in [c_S,\frac{\Lambda}{\mu_0}]$ for each $s\in[0,t]$ and from \eqref{Eq:estimates_vm} we deduce that $\Phi_s^S(v^{(m)})\in[c_S-\ep,\frac{\Lambda}{\mu_0}+\ep]$ for each $s\in[0,t]$ and every $m\geq m_{\ep}$. It shows both that
$$\left|S_0^*g\left(\dfrac{\Phi_s^S(v)}{S_0^*}\right)-S_0^*g\left(\dfrac{\Phi_s^S(v^{(m)})}{S_0^*}\right)\right|\leq \left|\Phi_s^S(v)-\Phi_s^S(v^{(m)})\right|\times \left(1+\dfrac{S_0^*}{c_S-\ep}\right)\leq \ep\left(1+\dfrac{S_0^*}{c_S-\ep}\right):=\kappa_1(\ep)$$
and that
\begin{equation*}
\begin{array}{rcl}
&&s\times \sup_{\xi\in[0,s]}\left|\dfrac{(\Lambda-\mu_S\Phi_{\xi}^S(v))^2}{\mu_S \Phi_{\xi}^S(v)}-\dfrac{(\Lambda-\mu_S\Phi_{\xi}^S(v^{(m)}))^2}{\mu_S \Phi_{\xi}^S(v^{(m)})}\right| \\
&\leq&t\ep\left(\dfrac{\Lambda}{\mu_S(c_S-\ep)}\left(2\Lambda\mu_S +\mu_S^2\left(\dfrac{\Lambda}{\mu_0}+\ep\right)\right) +\dfrac{\left(\Lambda+\mu_S\left(\dfrac{\Lambda}{\mu_0}+\ep\right)\right)^2}{c_S(c_S-\ep)}\right):=\kappa_2(\ep)
\end{array}
\end{equation*}
Using again \eqref{Eq:estimates_vm} we have for every $s\in[0,t]$ and every $m\geq m_{\ep}$ on one hand that
\begin{flalign*}
\left|\sum_{k=1}^n \int_0^\infty \Psi_k \Phi_s^k(v)(a)da-\sum_{k=1}^n \int_0^\infty \Psi_k \Phi_s^k(v^{(m)})(a)da\right|&\leq \sum_{k=1}^n \|\Psi_k\|_{L^\infty}\|\Phi_s^k(v^{(m)})-\Phi_s^k(v)\|_{L^1} \\
&\leq \ep\sum_{k=1}^n \|\Psi_k\|_{L^\infty}:=\kappa_3(\ep)
\end{flalign*}
leading to
$$\left|L_0(\Phi_s(v))-L_0(\Phi_s(v^{(m)}))\right|+\left|L_0(v)-L_0(v^{(m)})\right|\leq 2(\kappa_1(\ep)+\kappa_3(\ep))$$
and on the other hand that
$$\sum_{k=1}^n \left(\dfrac{\RR_{0,k}-1}{r_k}\right) s\times \|\beta_k\|_{L^\infty} \sup_{\xi\in[0,s]}\|\Phi_{\xi}^k(v)-\Phi_{\xi}^k(v^{(m)})\|\leq \ep\sum_{k=1}^n \left(\dfrac{\RR_{0,k}-1}{r_k}\right) t\times \|\beta_k\|_{L^\infty}:=\kappa_4(\ep).$$
Since $\kappa_k(\ep)\underset{\ep\to 0}{\longrightarrow}0$ for each $k\in\llbracket 1,4\rrbracket$ then clearly \eqref{Eq:L0_v} is proved for each $s\in[0,t]$ whence \eqref{Eq:L0_int} and consequently  \eqref{Eq:L0_deriv} hold for every $v\in \A$. Finally, if $v\in \A \cap \left(\bigcap_{j=1}^{\sigma_{n_>}}\partial \S_j\right)$, it is clear that $\frac{dL_0(\Phi_t(v))}{dt}\leq 0$ for every $t\geq 0$, hence $L_0$ is a Lyapunov functional on $\A \cap \left(\bigcap_{j=1}^{\sigma_{n_>}}\partial \S_j\right)$.
\end{proof}

\begin{lemma}\label{Lemma:L0-constant}
Let $v\in \A\cap\left( \bigcap_{j=1}^{\sigma_{n_>}}\partial \S_j\right)$ and $\gamma(v):=\{\psi(t),t\in \R\}\subset \Gamma(v)\subset \A$ be a complete orbit through $v$. If $L_0$ is constant on $\gamma(v)$ then $\gamma(v)=\{E_0\}$. In particular, $v=\{E_0\}$.
\end{lemma}

\begin{proof}
Let $v:=(S_0,x_{1,0},...,x_{n,0})\in \A\cap\left( \bigcap_{j=1}^{\sigma_{n_>}}\partial \S_j\right)$ and $\gamma(v):=\{\psi(t), t\in \R\}\subset \Gamma(v)\subset \A $ be a complete orbit through $v$. Suppose that $L_0$ is constant on $\gamma(v)$. It implies that
$$\dfrac{d}{dt}L_0(\Phi_t(\psi(s)))=0, \quad \forall (t,s)\in\R_+^*\times \R.$$
From \eqref{Eq:L0_deriv} we get $\Phi_t^S(\psi(s))=S_0^*$ for every $(t,s)\in\R_+\times \R$ (in particular $S_0=S_0^*$). Using the $S$-equation of \eqref{Eq:model}, we deduce that
$$\int_0^\infty \beta_k(a)\Phi_t^k(\psi(s))(a)da=0, \quad \forall (t,s,k)\in \R^*_+\times \R\times \llbracket 1,n\rrbracket.$$
With \eqref{Eq:Duhamel_phik} (second equation) we get
$$\Phi_t^k(\psi(s))(a)=0, \quad \forall (t,s,a)\in \R_+\times \R\times [0,t].$$
We remark that 
$$\Phi_t^k(\psi(s))(a)=\Phi_{t+a}^k(\psi(s-a))(a)=0, \quad \forall (t,s,a)\in\R_+\times \R\times [t,+\infty)$$
so it follows that $x_{k,0}\equiv 0$ for every $k\in\llbracket 1,n\rrbracket$, whence $v=\{E_0\}$ and result is proved.
\end{proof}

\subsection{Case $\RR_{0,1}\leq 1$}

We suppose here that $n_>=0$ which is equivalent to $\RR_{0,k}\leq 1$ for each $k\in\llbracket 1,n\rrbracket$. We now show that $\A=\{E_0\}$. Let $z\in \A$. By invariance of $\A$ there exists a complete orbit $\gamma(z)=\{\phi(s),s\in \R\}\subset \A$ through $z$. From Lemma \ref{Lemma:L0} we know that $L_0(\phi(t))$ is well-defined for each $t\in \R$. Also the equation \eqref{Eq:L0_deriv} is satisfied for every $v\in \gamma(z)$. The fact that $\RR_{0,k}\leq 1$ for each $k\in\llbracket 1,n\rrbracket$ implies that $L_0$ is a Lyapunov function on $\A$ (and in particular on $\gamma(z)$). Using \cite[Proposition 2.51 p. 53]{SmithThieme2011} we deduce that $L_0$ is constant on $\alpha(\phi)$ and on $\omega(z)$. Since $\gamma(z)\subset \A$ which is compact, then $\gamma^-(z):=\{\phi(s),s\leq 0\}$ is relatively compact on $X$ and non empty. From \cite[Theorem 2.48, p.52]{SmithThieme2011} the alpha-limit set $\alpha(\phi)$ is nonempty, compact, invariant, connected and $\lim_{t\to -\infty}d(\phi(t),\alpha(\phi))=0$. Since $L_0$ is constant on $\alpha(\phi)$ then it is constant on every complete orbit $\gamma(v)\subset \alpha(\phi)$ and by Lemma \ref{Lemma:L0-constant} we get $\alpha(\phi)=\{E_0\}$. It follows that
$$\lim_{s\to -\infty}d(\phi(s),\{E_0\})=0.$$
Using the facts that $L_0(E_0)=0$ and that $L_0$ is a Lyapunov function on $\gamma(z)$, then we necessarily get $L_0(\phi(s))=0$ for every $s\in\R$.
Finally $\{\phi(s),s\in \R\}\subset \Gamma(z)\subset \A$ is a complete orbit through $z\in \A$ on which $L_0$ is constant. Using again Lemma \ref{Lemma:L0-constant} we get $z=\{E_0\}$. In conclusion $\A=\{E_0\}$ which is GAS on $X_+$ and the first item is proved.

\subsection{Case $\RR_{0,1}>1$}

In this section, we suppose that $n_>\geq 1$. It remains to prove the second item for $k=1$. We first recall the definitions of the different notions of persistence.
\begin{definition}[{\cite[Def 3.1 p. 61]{SmithThieme2011}}]
Let $\p:X_+\to \R_+$ be a continuous function non identically zero. The semiflow $\Phi$ is:
\begin{enumerate}[label=(\arabic*)]
\item weakly $\p$-persistent if
$$\forall x\in X: \p(x)>0, \quad \limsup_{t\to +\infty}\p(\Phi_t(x))>0,$$
\item (strongly) $\p$-persistent if
$$\forall x\in X: \p(x)>0, \quad \liminf_{t\to +\infty}\p(\Phi_t(x))>0,$$
\item uniformly weakly $\p$-persistent if $\exists \ \ep>0$:
$$\forall x\in X : \p(x)>0, \quad \limsup_{t\to +\infty}\p(\Phi_t(x))\geq \ep,$$
\item uniformly (strongly) $\p$-persistent if $\exists \ \ep>0$:
$$\forall x\in X: \p(x)>0, \quad \liminf_{t\to +\infty}\p(\Phi_t(x))\geq \ep.$$
\end{enumerate}
\end{definition}

\subsubsection{Existence of a global attractor}
\label{Sec:existence_A0}

We start with the following lemma.

\begin{lemma}\label{Lemma:L_k-Lyap}
Suppose that $n_>\geq 2$ and let $k\in\llbracket 2,n_>\rrbracket$. Let $\J\subset \llbracket 1+\sigma_{k-1},\sigma_k\rrbracket$ with $\J\neq \emptyset$. Let
$$z\in\left(\bigcap_{j=1}^{\sigma_{k-1}} \partial \S_j\right)\cap \left(\bigcap_{j\in \J} \S_j\right)\cap \left(\bigcap_{j\in \llbracket 1+\sigma_{k-1},\sigma_k\rrbracket\setminus \J} \partial \S_j\right)$$
and let
$$\gamma(z)=\{\phi(t),t\in\R\}\subset \left(\bigcap_{j=1}^{\sigma_{k-1}} \partial \S_j\right)\cap \left(\bigcap_{j\in \J} \S_j\right)\cap \left(\bigcap_{j\in \llbracket 1+\sigma_{k-1},\sigma_k\rrbracket\setminus \J} \partial \S_j\right)$$
be a complete orbit through $z$. Let $(\alpha_{1+\sigma_{k-1}},...,\alpha_{\sigma_k})\in[0,1]^{\sigma_k-\sigma_{k-1}}$ such that $\sum_{j=1+\sigma_{k-1}}^{\sigma_k}\alpha_j=1$ and with $\alpha_j=0$ for every $j\in\llbracket 1+\sigma_{k-1},\sigma_k\rrbracket \setminus \J$. Let $\J^*=\{j\in\J: \alpha_j>0\}$. Then:
\begin{enumerate}
\item The functional $L_k^{\alpha_{1+\sigma_{k-1}},...,\alpha_{\sigma_k}}$ defined by \eqref{Eq:Lk} is a Lyapunov functional on the complete orbit $\gamma(z)$ and for each $(t,v)\in \R_+^*\times \gamma(z)$ we have:
\begin{equation}
    \begin{array}{rll}
    \label{Eq:Lk_deriv-inf}
       &\dfrac{dL_k^{\alpha_{1+\sigma_{k-1}},...,\alpha_{\sigma_k}}(\Phi_{t}(v))}{dt}=\displaystyle -\sum_{j\in\J^*}S^*_{\sigma_k}\int_0^\infty \beta_j(a)x^*_{j,\alpha_j}(a)g\left(\dfrac{\Phi_t^j(v)(a)\int_0^\infty \beta_j(s)x_{j,\alpha_j}^*(s)ds}{x_{j,\alpha_j}^*(a)\int_0^\infty \beta_j(s)\Phi_t^j(v)(s)ds}\right)da \\
       &\qquad -|\J^*| \displaystyle g\left(\dfrac{S^*_{\sigma_k}}{\Phi_t^S(v)}\right)- \sum_{j=1+\sigma_k}^n \int_0^\infty \beta_j(a)\Phi_t^j(v)(a)\left(\dfrac{1}{r_j}-S^*_{\sigma_k}\right)da-\dfrac{\mu_S}{\Phi_t^S(v)}(\Phi_t^S(v)-S^*_{\sigma_k})^2.
    \end{array}
\end{equation}
\item Suppose that $\J^*=\J$, that is for each $j\in \J$ we have $\alpha_j>0$. Suppose also that $L_k^{\alpha_{1+\sigma_{k-1}},...,\alpha_{\sigma_k}}$
is constant on $\gamma(z)$. Then $\gamma(v)\subset \E_{k,\J}$ (which is defined by \eqref{Eq:equilibria_kJ}). In particular $v\in \E_{k,\J}$.
\end{enumerate}
\end{lemma}

\begin{proof}
\begin{enumerate}
\item First we suppose that the assumptions of the lemma hold and let $v:=(S_0,x_{1,0},...,x_{n,0})\in\gamma(z)$. Since $\gamma(z)\subset \partial \S_1$ then $\int_0^\infty \Psi_1(a)x_{1,0}(a)da=0$ and $x_{1,0}\equiv 0$. It follows that $L_k^{\alpha_{1+\sigma_{k-1}},...,\alpha_{\sigma_k}}(v)=\tilde{L}_k^{\alpha_{1+\sigma_{k-1}},...,\alpha_{\sigma_k}}(\tilde{v})$
for $\tilde{v}:=(S_0,x_{2,0},...,x_{n,0})$ where
\begin{flalign*}
\tilde{L}_k^{\alpha_{1+\sigma_{k-1}},...,\alpha_{\sigma_k}}(\tilde{v})=&S_{\sigma_k}^*g\left(\dfrac{S_0}{S_{\sigma_k}^*}\right)+\sum_{j=2}^n \int_0^\infty \eta_j\Psi_j(a)x_{j,0}(a)da \\
&+\sum_{j=2}^n\int_0^\infty (1-\eta_j) \Psi_j(a) x^*_{j,\alpha_j}(a)g\left(\dfrac{x_{j,0}(a)}{x^*_{j,\alpha_j}(a)}\right)da
\end{flalign*}
and with
$$\eta_j=\begin{cases}
1 & \text{ if } j\in\llbracket 2, \sigma_{k-1}\rrbracket\cup \llbracket 1+\sigma_k,n\rrbracket \ \text{ or if } \alpha_j=0 \\
0 & \text{else.}
\end{cases}$$
We define $\tilde{\Phi}:(t,\tilde{v})\longmapsto (\Phi_t^S(v),\Phi_t^2(v),...,\Phi_t^n(v))$ that is the solution of the reduced problem \eqref{Eq:model} with the infected populations $j\in\llbracket 2,n\rrbracket$ and initial condition $\tilde{v}$. By recurrence hypothesis the function
$$t\longmapsto \tilde{L}_k^{\alpha_{1+\sigma_{k-1}},...,\alpha_{\sigma_k}}(\tilde{\Phi}_t(\tilde{v}))$$
is well-defined on $\R_+$, continuous and satisfy
\begin{equation*}
    \begin{array}{rll}
       &\dfrac{dL_k^{\alpha_{1+\sigma_{k-1}},...,\alpha_{\sigma_k}}(\Phi_{t}(\tilde{v}))}{dt}=\displaystyle -\sum_{j\in\J^*}S^*_{\sigma_k}\int_0^\infty \beta_j(a)x^*_{j,\alpha_j}(a)g\left(\dfrac{\tilde{\Phi}_t^j(\tilde{v})(a)\int_0^\infty \beta_j(s)x_{j,\alpha_j}^*(s)ds}{x_{j,\alpha_j}^*(a)\int_0^\infty \beta_j(s)\tilde{\Phi}_t^j(\tilde{v})(s)ds}\right)da \\
       &\qquad -|\J^*| \displaystyle g\left(\dfrac{S^*_{\sigma_k}}{\tilde{\Phi}_t^S(\tilde{v})}\right)
       -\sum_{j=1+\sigma_k}^n  \int_0^\infty \beta_j(a)\tilde{\Phi}_t^j(\tilde{v})(a)\left(\dfrac{1}{r_j}-S^*_{\sigma_k}\right)da-\dfrac{\mu_S}{\tilde{\Phi}_t^S(\tilde{v})}(\tilde{\Phi}_t^S(\tilde{v})-S^*_{\sigma_k})^2
    \end{array}
\end{equation*}
which is non-positive since $\frac{1}{r_j}-S^*_{\sigma_k}=\frac{1}{r_j}-\frac{1}{r_{\sigma_k}}=\frac{\Lambda}{\mu_S}(\frac{1}{\RR_{0,j}}-\frac{1}{\RR_{0,\sigma_k}})>0$ for every $j\in\llbracket 1+\sigma_k, n\rrbracket$. It naturally follows that the first point is true.

\item \textbf{Step 1.} First we know that $\Phi^j_t(\phi(s))(a)=0$ for every $j\in\llbracket 1,\sigma_{k-1}\rrbracket \cup (\llbracket 1+\sigma_{k-1},\sigma_k\rrbracket\setminus \J)$ and every $(s,t,a)\in\R\times (\R_+)^2$ since $\gamma(z)\subset \left(\bigcap_{j=1}^{\sigma_{k-1}} \partial \S_j\right)\cap \left(\bigcap_{j\in \llbracket 1+\sigma_{k-1},\sigma_k\rrbracket\setminus \J} \partial \S_j\right)$ and in particular $x_{j,0}\equiv 0$. Suppose now that $L_k^{\alpha_{1+\sigma_{k-1}},...,\alpha_{\sigma_k}}$ is constant on $\gamma(v)$. It implies that
$$\dfrac{d}{dt}L_k^{\alpha_{1+\sigma_{k-1}},...,\alpha_{\sigma_k}}(\Phi_t(\phi(s)))=0, \quad \forall t>0, \quad \forall s\in\R.$$
From \eqref{Eq:Lk_deriv-inf} we get $\Phi_t^S(\phi(s))=S^*_{\sigma_k}$ for every $(t,s)\in\R_+\times \R$ (in particular $S_0=S^*_{\sigma_k}$). Using \eqref{Eq:Lk_deriv-inf} and Assumption \ref{Assum:1} we deduce that
$$\forall (j,t)\in\J^*\times \R_+, \quad \exists \ c_j(t)>0: \quad 
\Phi_t^j(\phi(s))(a)=c_j(t) x^*_{j,\alpha_j}(a), \quad  \forall s\in \R \ \text{ a.e. } a\in[\underline{\beta_j},\overline{\beta_j}).$$
Let $(j,s,a)\in\J^*\times\R\times[\underline{\beta_j},\overline{\beta_j})$.
Using the fact that $\gamma(z)$ is a complete orbit we have
$$c_j(t)x^*_{j,\alpha_j}(a)=\Phi_t^j(\phi(s))(a)=\Phi_{\tilde{t}}^j(\phi(s+t-\tilde{t}))(a)=c_j(\tilde{t})x^*_{j,\alpha_j}(a), \quad \forall (t,\tilde{t})\in (\R_+)^2 \ \text{ a.e. }  a\in[\underline{\beta_j},\overline{\beta_j}).$$
whence $c_j$ does not depend on $t$ and we have
$$\forall j\in\J^*, \quad \exists \ c_j>0: \quad 
\Phi_t^j(\phi(s))(a)=c_jx^*_{j,\alpha_j}(a), \quad  \forall (s,t)\in \R\times \R_+ \ \text{ a.e. } a\in[\underline{\beta_j},\overline{\beta_j}).$$
From \eqref{Eq:Duhamel_phik} and from the latter equation with $t=a$ we see that
$$c_jx_{j,\alpha_j}^*(a)=\Phi^j_a(\phi(s))(a)=\Phi_{\xi}^j(\phi(s))(\xi)e^{-\int_{\xi}^a\mu_j(z)dz}, \quad \forall \xi\in[0,a]$$
and
$$\Phi_{\xi}^j(\phi(s))(\xi)=\Phi^j_a(\phi(s))(a)e^{-\int_a^{\xi}\mu_j(z)dz}=c_jx^*_{j,\alpha_j}(a)e^{-\int_a^{\xi}\mu_j(z)dz}=c_jx^*_{j,\alpha_j}(\xi), \quad \forall \xi\geq a.$$
These two equations lead to
$$\forall (j,s,a)\in \J^*\times \R\times \R_+: \quad \Phi^j_a(\phi(s))(a)=c_jx^*_{j,\alpha_j}(a).$$
Let $t\geq 0$. Since $\gamma(z)$ is a complete orbit then we have
$$\Phi_t^j(\phi(s))(a)=\Phi^j_a(\phi(t+s-a))(a)=c_jx^*_{j,\alpha_j}(a)=x^*_{j,c_j\alpha_j}(a)$$
and this latter formula is true for every $(j,s,t,a)\in\J^*\times\R\times(\R_+)^2$. In particular $x_{j,0}=x^*_{j,c_j\alpha_j}$ for each $j\in\J^*$. Now let $j\in\llbracket 1+\sigma_k,n\rrbracket$. From \eqref{Eq:Lk_deriv-inf} we get
$$\int_0^\infty \beta_j(a)\Phi_t^j(\phi(s))(a)da=0$$
for every $(s,t)\in\R\times \R_+$, whence $\Phi_t^j(\phi(s))(a)=0$ a.e. $a\in[\underline{\beta_j},\overline{\beta_j})$. As before, letting $t=a$ and using \eqref{Eq:Duhamel_phik} we get
$$0=\Phi_a^j(\phi(s))(a)=\Phi_{\xi}^j(\phi(s))(\xi)e^{-\int_{\xi}^a \mu_j(z)dz}, \quad \forall \xi\in[0,a]$$
and 
$$\Phi_{\xi}^j(\phi(s))(\xi)=\Phi_a^j(\phi(s))(a)e^{-\int_a^{\xi}\mu_j(z)dz}=0, \quad \forall \xi\geq a$$
whence $\Phi_a^j(\phi(s))(a)=0$ for every $(j,a,s)\in \llbracket 1+\sigma_k,n\rrbracket\times \R_+\times \R$. Let $t\geq 0$, since $\gamma(z)$ is a complete orbit then
$$\Phi_t^j(\phi(s))(a)=\Phi^j_a(\phi(t+s-a))(a)=0$$
and this latter formula is true for every $(j,s,t,a)\in \llbracket 1+\sigma_k,n\rrbracket\times \R\times (\R_+)^2$. In particular $x_{j,0}\equiv 0$.

\textbf{Step 2.} By the first step, looking at the $S$-equation of \eqref{Eq:model} we get
$$\displaystyle\sum_{j\in\J^*}\int_0^\infty \beta_j(a)x^*_{j,c_j\alpha_j}(a)da=\sum_{j\in\J^*}\int_0^\infty \beta_j(a)\Phi_t^j(\phi(s))(a)da=\dfrac{\Lambda}{S^*_{\sigma_k}}-\mu_S=\mu_S(\RR_{0,k}-1), \quad \forall (t,s)\in \R_+^*\times \R$$
which is equivalent to 
$$\sum_{j\in\J^*}c_j\mu_S(\RR_{0,j}-1)\alpha_j=\mu_S(\RR_{0,k}-1)$$
but since $\RR_{0,j}=\RR_{0,k}$ for every $j\in\J^*$ then we get the condition
$$\sum_{j\in\J^*}c_j\alpha_j=1$$
which implies $\phi(t)\in \E_{k,\J^*}=\E_{k,\J}$ for every $t\in \R$ and this ends the proof.
\end{enumerate}
\end{proof}

Now we prove that $\Phi$ has a global attractor $\A_0\subset \A$ in $\cup_{k=1}^{\sigma_1} \S_k$. We proceed in several stages. Let $\p:X_+\to \R_+$ be the function defined by:
$$\p(z)=\sum_{k=1}^{\sigma_1}\int_0^{\overline{\beta_k}}x_{k,0}(s)ds$$
for every $z=(S_0,x_{1,0},...,x_{n,0})\in X_+$. Let the subsets
$$\S_0=\{z\in X_+: \p(z)>0\}=\bigcup_{k=1}^{\sigma_1} \S_k, \qquad \partial \S_0=\{z\in \X_+: \p(z)=0\}=\bigcap_{k=1}^{\sigma_1} \partial \S_k.$$
We also define the set
$$\Omega:=\bigcup_{z\in \partial \S_0}\omega(z)=\{E_0\}\cup \left(\bigcup_{k=2}^{n_{>}} \E_k\right)$$
by recurrence hypothesis. We have the following properties:
\begin{enumerate}
    \item The set $\{E_0\}$ and each subset $\E_k$ ($k\in\llbracket 2,n_{>}\rrbracket$) are invariant and disjoints. Also, $\{E_0\}$ is clearly compact while $\E_k$ is compact as the image of the compact set
    $$\{(\alpha_1,...,\alpha_n)\in[0,1]^n: \sum_{j=1+\sigma_{k-1}}^{\sigma_k}\alpha_j=1, \alpha_j=0 \ \forall j\in\llbracket 1,\sigma_{k-1}\rrbracket \cup \llbracket 1+\sigma_k,n\rrbracket\}$$
    by the continuous function
    $$(\alpha_1,...,\alpha_n)\longmapsto E^{*,k}_{\alpha_1,...,\alpha_n}.$$
    
    
    
    \item $\textbf{(i)}$ Let $\ep>0$ such that 
    \begin{equation}\label{Eq:eps_M}
    \left(\dfrac{\Lambda}{\mu_S}-\ep\right)r_1>1
    \end{equation}
    Following the proof of \cite[Proposition 4.4.8.(a)]{Richard2020}, we show that
    \begin{equation}\label{Eq:M_eps}
   \forall z\in \M_{\ep}:=\{\overline{z}\in \S_0, \|\overline{z}-E_0\|_{X}\leq \ep\}, \quad \exists \ \overline{t}>0: \|\Phi_{\overline{t}}(z)-E_0\|_{X}>\ep
    \end{equation}
    which implies that
    $$\{\overline{z}\in X_+, \p(\overline{z})>0: \lim_{t\to +\infty}\Phi_t(\overline{z})=E_0\}=\emptyset$$
    meaning that the semiflow $\Phi$ is weakly persistent (we note that in \cite{Richard2020}, the equation \eqref{Eq:M_eps} was proved for $\overline{z}$ in the intersection, but the same proof works for the reunion).
    Indeed, suppose by contradiction that there exists $z:=(S_0,x_{1,0},...,x_{n,0})\in \S_0$ such that
    $$\|\Phi_t(z)-E_0\|_{X}\leq \ep, \ \forall t\geq 0$$
    then
    $$\Phi_t^S(z)\geq \dfrac{\Lambda}{\mu_S}-\ep, \ \forall t\geq 0.$$
    Let $j\in\llbracket1,\sigma_1\rrbracket$ such that $z\in \S_{j}$, then by continuity arguments there exists $c\in(\underline{\beta_{j}},\infty)$ such that
    $$\left(\dfrac{\Lambda}{\mu_S}-\ep\right) \int_0^c \beta_{j}(a)\pi_{j}(a)da>1.$$
    Also, using \cite[Proposition 4.4.4]{Richard2020} we know that there exists $\tau \geq 0$ such that
    $$\int_0^\infty \beta_{j}(a)\Phi_t^{j}(z)(a)da>0$$
    for every $t\geq \tau$. Thus we arrive at
    \begin{equation*}
        \left\{
        \begin{array}{rcl}
        \dfrac{\partial \Phi_t^j(z)(a)} {\partial t}+\dfrac{\partial \Phi_t^j(z)(a)} {\partial a}&=&-\mu_{j}(a)\Phi_t^j(z)(a), \\
        \Phi_t^j(z)(0)&\geq& \left(\dfrac{\Lambda}{\mu_S}-\ep\right)\displaystyle\int_0^c \beta_{j}(a)\Phi_t^j(z)(a)da
        \end{array}
        \right.
    \end{equation*}
    for every $t\geq \overline{\tau}>\tau$. It follows that $\Phi_{\overline{\tau}}^{j}(z)\in L^1_+(0,c)\setminus \{0\}$ since 
    $$\Phi_{\overline{\tau}}^{j}(z)(a)=\left(\int_0^\infty \beta_{j}(\xi)\Phi^{j}_{\overline{\tau}-a}(z)(\xi)d\xi\right)e^{-\int_0^a \mu_{j}(s)ds}>0, \ \forall a\in[0,\overline{\tau}-\tau]$$
    and then $\lim_{t\to \infty}\int_0^c \Phi_t^j(z)(a)da=\infty$ (see \textit{e.g.} \cite[Lemma 4.3]{Richard2020}) by using the comparison theorem \cite[Proposition 5.2]{MagalSeydi2019} which is absurd, hence \eqref{Eq:M_eps} is true.
    
    $\textbf{(ii)}$ Now, let $k\in \llbracket 2,n_{>}\rrbracket$ (hence we suppose that $n_>\geq 2$). Let $\ep>0$ such that
    \begin{equation}\label{Eq:eps_Mk}
    \left(\dfrac{1}{r_{\sigma_k}}-\ep\right)r_1>1.
    \end{equation}
    As the first point, we can prove that
    \begin{equation}\label{Eq:Mk_eps}
   \forall z\in \M^k_{\ep}:=\{\overline{z}\in \S_0, \|\overline{z}-\E_k\|_{X}\leq \ep\}, \quad \exists \ \overline{t}>0: \|\Phi_{\overline{t}}(z)-\E_k\|_{X}>\ep.
    \end{equation}
    Indeed, suppose by contradiction that there exists $z:=(S_0,x_{1,0},...,x_{n,0})\in \S_0$ such that
    $$\Phi_t^S(z)\geq S^*_{\sigma_k}-\ep=\dfrac{1}{r_{\sigma_k}}-\ep, \ \forall t\geq 0.$$
    Let $j\in\llbracket1,\sigma_1\rrbracket$ such that $z\in \S_{j}$ then by continuity arguments (since $r_j=r_1$ because $\RR_{0,j}=\RR_{0,1}$) there exists $c\in(\underline{\beta_{j}},\infty)$ such that
    $$\left(\dfrac{1}{r_k}-\ep\right) \int_0^c \beta_{j}(a)\pi_{j}(a)da>1$$
    and the same arguments as before lead to $\lim_{t\to \infty}\int_0^c \Phi_t^j(z)(a)da=\infty$ which is absurd, hence \eqref{Eq:Mk_eps} is true.
    
    
    \item $\textbf{(i)}$ We prove that $\{E_0\}$ is isolated \textit{i.e.} the maximal compact invariant set of a neighborhood of itself (this neighborhood being called an isolating neighborhood). Indeed, let $\ep>0$ be small enough such that \eqref{Eq:eps_M} is true then following the arguments of \cite[Proposition 2.4]{MagCluskWebb2010} (see also \cite{HaleWaltman89}), we can prove that the ball centered at $E_0$ and radius $\ep$, denoted by $B_{X}(E_0,\ep)$, is an isolating neighborhood of $E_0$. To prove it we let $K\subset B_{X}(E_0,\ep)$ be a compact invariant subset of $B_{X}(E_0,\ep)$ and let $z\in K$. We need to prove that $K=\{E_0\}$. If $z\in \M_{\ep}$ then clearly, from \eqref{Eq:M_eps}, it follows that $K$ is not invariant, whence $z\not\in \M_{\ep}$ \textit{i.e.} $z\in \partial \S_0$. First we prove that
    \begin{equation}\label{Eq:K_subset}
    K\subset \bigcap_{k\in\llbracket 1,\sigma_{n_>}\rrbracket}\partial \S_k.
    \end{equation}
    If $n_>=1$, then \eqref{Eq:K_subset} is true by definition of $\partial \S_0$. Suppose now that $n_>\geq 2$, that $z\in\cup_{k\in\llbracket 1+\sigma_1,\sigma_{n_>}\rrbracket}\S_k$ and also that
    $$\ep<S_0^*-\max_{k\in\llbracket 1,\sigma_{n_>}\rrbracket}\left\{\dfrac{1}{r_{\sigma_k}}\right\}=\dfrac{\Lambda}{\mu_S}\left(1-\dfrac{1}{\min_{k\in\llbracket 1,\sigma_{n_>}\rrbracket}\{\RR_{0,\sigma_k}\}}\right).$$
    Let $k=\min\{j\in\llbracket 2,n_>\rrbracket: z\in \cup_{s\in\llbracket 1+\sigma_{j-1},\sigma_j\rrbracket} \S_s\}$. By invariance of $K$ and since $z\in \partial \S_0\subset \partial \S_1$, then denoting $z=(S_0,x_{1,0},...,x_{n,0})$ we necessarily have $x_{1,0}\equiv 0$ and $\Phi_t^1(z)=0$ for every $t\geq 0$. It follows that the function $t\longmapsto (\Phi^S_t(z),\Phi_t^2(z),...,\Phi_t^n(z))$ is solution of the reduced problem \eqref{Eq:model} with the infected populations $k\in\llbracket 2,n\rrbracket$. By recurrence hypothesis we deduce that $\Phi_t(z)\underset{t\to \infty}{\longrightarrow} \E_k$ hence $\Phi_t^S(z)\underset{t\to \infty}{\longrightarrow} \frac{1}{r_{\sigma_k}}$. This is absurd since $S_0^*-\frac{1}{r_{\sigma_k}}>\ep$ (thus $\Phi_t(z)\not\in B_X(E_0,\ep)$ for $t$ large enough and consequently $K$ is not invariant). It follows that \eqref{Eq:K_subset} is true.
    
     Since $K$ is invariant, there exists a complete orbit $\gamma(z)=\{\phi(t),t\in \R\}$ through $z$. Using Lemma \ref{Lemma:L0} and \eqref{Eq:K_subset} we deduce that $L_0$ is a Lyapunov functional on $\gamma(z)$. As in the case $\RR_{0,1}\leq 1$ we deduce first that $L_0$ is constant on $\alpha(\phi)$ and by using Lemma \ref{Lemma:L0-constant} that $\alpha(\phi)=\{E_0\}$. Secondly, we deduce that $L_0$ is constant on $\gamma(z)$ which leads to $z=\{E_0\}$ by Lemma \ref{Lemma:L0-constant}. Finally $K=\{E_0\}$ and is isolated.

$\textbf{(ii)}$ Suppose that $n_>\geq 2$ and let $k\in \llbracket 2,n_{>}\rrbracket$. We prove that $\E_k$ is isolated. Let $\ep>0$ be small enough such that \eqref{Eq:eps_Mk} is satisfied and such that
$$S_0^*-\frac{1}{r_k}>\ep, \qquad \min_{j\in\llbracket 1,n_{>}\rrbracket\setminus{\{k\}}}\left\{\left|\frac{1}{r_j}-\frac{1}{r_k}\right|\right\}>\ep$$

Let $K\subset B_{X}(\E_k,\ep)$ be a compact invariant subset and let $z\in K$. If $z\in \M_{\ep}^k$ then from \eqref{Eq:Mk_eps} it follows that $K$ is not invariant, whence $z\not\in \M^k_{\ep}$ \textit{i.e.} $z\in \partial \S_0$. 

\underline{Step 1.} We prove that
\begin{equation}\label{Eq:subset_inter}
z\in \left(\bigcap_{j=1}^{\sigma_{k-1}} \partial \S_j\right)\cap \left(\bigcup_{j=1+\sigma_{k-1}}^{\sigma_k}\S_j\right).
\end{equation}
We already know that $z\in \partial \S_0$. If $z\in \cup_{j\in\llbracket 1+\sigma_1,\sigma_{k-1}\rrbracket} \S_j$ then by recurrence hypothesis we would have $\Phi_t(z)\underset{t\to \infty}{\longrightarrow} \E_{j^*}$ for a $j^*\in\llbracket 2,k-1\rrbracket$. Thus $\Phi_t^S(z)\underset{t\to \infty}{\longrightarrow} \frac{1}{r_{j^*}}$ which is absurd since $|\frac{1}{r_{j^*}}-\frac{1}{r_k}|>\ep$. Suppose now that $z\in \cap_{j\in\llbracket 1+\sigma_{k-1},\sigma_k\rrbracket} \partial \S_j$, then either $\Phi_t(z)\underset{t\to \infty}{\longrightarrow} \E_{j^*}$ for $j^*\in \llbracket k+1,n_>\rrbracket$ which is absurd with the arguments above, or $\Phi_t(z)\underset{t\to \infty}{\longrightarrow} E_0$ which implies $\Phi_t^S(z)\underset{t\to \infty}{\longrightarrow} S_0^*$ and this is also absurd since $S_0^*-\frac{1}{r_k}>\ep$. It follows that \eqref{Eq:subset_inter} is satisfied.

\underline{Step 2.} Let $\J\subset \llbracket 1+\sigma_{k-1},\sigma_k\rrbracket$ be the (non-empty) subset such that
$$z\in \left(\bigcap_{j\in \J} \S_j\right)\cap \left(\bigcap_{j\in \llbracket 1+\sigma_{k-1},\sigma_k\rrbracket\setminus \J} \partial \S_j\right).$$
Since $K$ is invariant, there exists a complete orbit $\gamma(z)=\{\phi(t), t\in\R\}$ through $z$. Still by invariance we have $\phi(t)\in \cap_{j\in \J} \S_j$ and $\phi(t)\in \cap_{j\in \llbracket 1,\sigma_k\rrbracket\setminus \J} \partial \S_j$ for every $t\in \R$.
Let $(\alpha_{1+\sigma_{k-1}},...,\alpha_{\sigma_k})\in[0,1]^{\sigma_k-\sigma_{k-1}}$ such that $\sum_{j=1+\sigma_{k-1}}^{\sigma_k}\alpha_j=1$, with $\alpha_j=0$ for every $j\in\llbracket 1+\sigma_{k-1},\sigma_k\rrbracket \setminus \J$ and with $\alpha_j>0$ for every $j\in\J$ (for example $\alpha_j=\frac{1}{|\J|}$ for $j\in\J$). We know by Lemma \ref{Lemma:L_k-Lyap} that $L_k^{\alpha_{1+\sigma_{k-1}},...,\alpha_{\sigma_k}}$ defined by \eqref{Eq:Lk} is a Lyapunov functional on $\gamma(z)$. A consequence of \cite[Proposition 2.51, p. 53]{SmithThieme2011} is that $L_k^{\alpha_{1+\sigma_{k-1}},...,\alpha_{\sigma_k}}$ is constant on $\alpha(\phi)$ and on $\omega(z)$. Following the arguments used in the case $\RR_{0,1}\leq 1$ we deduce that $\alpha(\phi)$ is nonempty, compact, invariant and $\lim_{t\to -\infty}d(\phi(t),\alpha(\phi))=0$. Since $L_k^{\alpha_{1+\sigma_{k-1}},...,\alpha_{\sigma_k}}$ is constant on every complete orbit $\gamma(v)\subset \alpha(\phi)$ then we can use Lemma \ref{Lemma:L_k-Lyap} to obtain $\alpha(\phi)\subset \E_{k,\J}$. Similarly we get $\omega(z)\subset \E_{k,\J}$ and then
$$\lim_{t\to -\infty}d(\phi(t),\E_{k,\J})=\lim_{t\to +\infty}d(\phi(t),\E_{k,\J})=0.$$

\underline{Step 3.} We prove now that there exists $(E^{*,k}_{\tau_1,...,\tau_n},E^{*,k}_{\omega_1,...,\omega_n})\in (\E_{k,\J})^2$ such that:
\begin{equation}\label{Eq:Phi_limit}
\lim_{t\to -\infty}\phi(t)=E^{*,k}_{\tau_1,...,\tau_n} \qquad \lim_{t\to +\infty}\phi(t)=E^{*,k}_{\omega_1,...,\omega_n}.
\end{equation}
Suppose by contradiction that there exist $(E^{*,k}_{\tau_1,...,\tau_n},E^{*,k}_{\tau'_1,...,\tau'_n})\in (\alpha(\phi))^2$ such that $(\tau_1,...,\tau_n)\neq (\tau'_1,...,\tau'_n)$. By definition of $\E_{k,\J}$ we have 
$$\tau_j=\tau'_j=0 \quad \forall j\in\llbracket 1,n\rrbracket\setminus \J$$
thus we arrive at the assumption
$$(\tau_{1+\sigma_{k-1}},...,\tau_{\sigma_k})\neq (\tau'_{1+\sigma_{k-1}},...,\tau'_{\sigma_k}).$$

Firstly we prove that for every $j\in\J$ we have $\tau_j=0 \Longleftrightarrow \tau'_j=0$. By contradiction suppose that there exists $j\in\J$ such that $\tau_j=0$ and $\tau'_j\neq 0$. Then we use the functional $L_k^{\alpha_{1+\sigma_{k-1}},...,\alpha_{\sigma_k}}$ defined by \eqref{Eq:Lk} where $\alpha_j=1$ and $\alpha_s=0$ for every $s\in\llbracket 1+\sigma_{k-1},\sigma_k\rrbracket\setminus \{j\}$. Using Lemma \ref{Lemma:L_k-Lyap} this is a Lyapunov functional on $\gamma(z)$, and as explained in Step 2, it should be constant on $\alpha(\phi)$ implying that
$$L_k^{\alpha_{1+\sigma_{k-1}},...,\alpha_{\sigma_k}}\left(E^{*,k}_{\tau_1,...,\tau_n}\right)=L_k^{\alpha_{1+\sigma_{k-1}},...,\alpha_{\sigma_k}}\left(E^{*,k}_{\tau'_1,...,\tau'_n}\right)$$
since $(E^{*,k}_{\tau_1,...,\tau_n},E^{*,k}_{\tau'_1,...,\tau'_n})\in (\alpha(\phi))^2$. On the other hand we have
$$L_k^{\alpha_{1+\sigma_{k-1}},...,\alpha_{\sigma_k}}(E^{*,k}_{\tau_1,...,\tau_n})=
\int_0^\infty \Psi_j(a)\dfrac{\mu_S}{r_j}\left(\RR_{0,j}-1\right)\pi_j(a)g(\tau_j)da+\hspace{-0.1cm}\sum_{s\in\J\setminus\{j\}} \int_0^\infty \Psi_s(a)x^*_{s,\tau_s}(a)da=\infty$$
and
$$L_k^{\alpha_{1+\sigma_{k-1}},...,\alpha_{\sigma_k}}(E^{*,k}_{\tau'_1,...,\tau'_n})=
\int_0^\infty \Psi_j(a)\dfrac{\mu_S}{r_j}\left(\RR_{0,j}-1\right)\pi_j(a)g(\tau'_j)da+\hspace{-0.1cm}\sum_{s\in\J\setminus\{j\}} \int_0^\infty \Psi_s(a)x^*_{s,\tau'_s}(a)da<\infty$$
which is absurd (the same proof works for $\tau_j\neq 0$ and $\tau'_j=0$).

Secondly we prove that for every $j\in \J$ we have $\tau_j=\tau'_j$. For this, we decompose the set $\J$ as $\J=\J_=\sqcup \J_> \sqcup \J_<$ where
$$\J_=:=\{j\in \J: \tau_j=\tau'_j\}, \qquad \J_{>}:=\{j\in \J: \tau_j>\tau'_j\}, \qquad \J_{<}:=\{j\in \J: \tau_j<\tau'_j\}.$$
We see that
$$\J_=\neq \J,\quad \J_>\neq \J, \quad \J_<\neq \J,  \quad \J_>\neq \emptyset,  \quad \J_<\neq \emptyset$$
since $(\tau_1,...,\tau_n)\neq (\tau'_1,...,\tau'_n)$ and $\sum_{j\in \J} \tau_j=\sum_{j\in \J}\tau'_j=1$. We use the functional $L_k^{\alpha_{1+\sigma_{k-1}},...,\alpha_{\sigma_k}}$ defined by \eqref{Eq:Lk} where for every $j\in\llbracket 1+\sigma_{k-1},\sigma_k\rrbracket$ we put
$$\alpha_j=\begin{cases}
\dfrac{\tau_j}{\sum_{s\in \J_>}\tau_s} & \text{ if } j\in \J_> \\
0 & \text{ otherwise}.
\end{cases}$$
so that $\sum_{j\in \J_>}\alpha_j=1$. Then we compute
\begin{flalign*}
L_k^{\alpha_{1+\sigma_{k-1}},...,\alpha_{\sigma_k}}\left(E^{*,k}_{\tau_1,...,\tau_n}\right)=&
\sum_{j\in \J_>}\int_0^\infty \Psi_j(a)\dfrac{\mu_S}{r_j}\left(\RR_{0,j}-1\right)\alpha_j\pi_j(a)g\left(\dfrac{\tau_j}{\alpha_j}\right)da \\
&+\sum_{j\in\J\setminus \J_>} \int_0^\infty \Psi_j(a)x^*_{j,\tau_j}(a)da
\end{flalign*}
and
\begin{flalign*}
L_k^{\alpha_{1+\sigma_{k-1}},...,\alpha_{\sigma_k}}\left(E^{*,k}_{\tau'_1,...,\tau'_n}\right)=&
\sum_{j\in \J_>}\int_0^\infty \Psi_j(a)\dfrac{\mu_S}{r_j}\left(\RR_{0,j}-1\right)\alpha_j\pi_j(a)g\left(\dfrac{\tau'_j}{\alpha_j}\right)da\\
&+\sum_{j\in\J \setminus \J_>} \int_0^\infty \Psi_j(a)x^*_{j,\tau'_j}(a)da.
\end{flalign*}
On one hand we have $\tau_j>\tau'_j$ (and necessarily $\tau'_j>0$ by the first point) for each $j\in\J_>$ leading to
$$g\left(\dfrac{\tau_j}{\alpha_j}\right)=g\left(\sum_{s\in \J_>}\tau_s\right)<g\left(\dfrac{\tau'_j}{\tau_j}\sum_{s\in \J_>}\tau_s\right)=g\left(\dfrac{\tau'_j}{\alpha_j}\right)$$
since $g$ is decreasing on $(0,1]$ and by using the following inequality:
$$1\geq \sum_{s\in \J_>}\tau_s>\dfrac{\tau'_j}{\tau_j}\sum_{s\in\J_>}\tau_s\geq 0.$$
On the other hand we see that
$$\int_0^\infty \Psi_j(a)x^*_{j,\tau_j}(a)da\leq \int_0^\infty \Psi_j(a)x^*_{j,\tau'_j}(a)da$$
for every $j\in\J\setminus \J_>=\J_=\sqcup \J_<$. It follows that 
$$L_k^{\alpha_{1+\sigma_{k-1}},...,\alpha_{\sigma_k}}\left(E^{*,k}_{\tau_1,...,\tau_n}\right)<L_k^{\alpha_{1+\sigma_{k-1}},...,\alpha_{\sigma_k}}\left(E^{*,k}_{\tau'_1,...,\tau'_n}\right)$$
which is absurd since $L_k^{\alpha_{1+\sigma_{k-1}},...,\alpha_{\sigma_k}}$ is a Lyapunov functional on $\gamma(z)$ by Lemma \ref{Lemma:L_k-Lyap} and as explained before it should be constant on $\alpha(\phi)$ that is to say
$$L_k^{\alpha_{1+\sigma_{k-1}},...,\alpha_{\sigma_k}}\left(E^{*,k}_{\tau_1,...,\tau_n}\right)=L_k^{\alpha_{1+\sigma_{k-1}},...,\alpha_{\sigma_k}}\left(E^{*,k}_{\tau'_1,...,\tau'_n}\right).$$

We deduce from this that $\alpha(\phi)$ is reduced to a single equilibrium, that is:
$$\alpha(\phi)\cap \E_{k,\J}=\left\{E^{*,k}_{\tau_1,...,\tau_n}\right\}$$
Similarly, we may prove that 
$$\omega(z)\cap \E_{k,\J}=\left\{E^{*,k}_{\omega_1,...,\omega_n}\right\}.$$
Thus \eqref{Eq:Phi_limit} is satisfied and the complete orbit $\gamma(z)$ is either a heteroclinic or a homoclinic orbit.

\underline{Step 4}. We now prove that 
\begin{equation}\label{Eq:Omega=tau}
E^{*,k}_{\tau_1,...,\tau_n}=E^{*,k}_{\omega_1,...,\omega_n}
\end{equation}
which amounts to prove that $\tau_j=\omega_j$ for every $j\in \J$. Similarly to the third step, we define the sets
$$\J_>:=\{j\in \J: \tau_j>\omega_j\}, \qquad \J_<:=\{j\in \J: \tau_j<\omega_j\}.$$
Suppose that $\J_>\neq \emptyset$. We use the functional $L_k^{\alpha_{1+\sigma_{k-1}},...,\alpha_{\sigma_k}}$ defined by \eqref{Eq:Lk} where for every $j\in\llbracket 1+\sigma_{k-1},\sigma_k\rrbracket$ we put
$$\alpha_j=\begin{cases}
\dfrac{\tau_j}{\sum_{s\in \J_>}\tau_s} & \text{ if } j\in \J_> \\
0 & \text{ otherwise}.
\end{cases}$$
so that $\sum_{j\in \J_>}\alpha_j=1$ and $\alpha_j>0$ for every $j\in\J_>$ since $\tau_j>\omega_j\geq 0$.
The same computations as in Step 3 lead to
$$L_k^{\alpha_{1+\sigma_{k-1}},...,\alpha_{\sigma_k}}\left(E^{*,k}_{\tau_1,...,\tau_n}\right)<L_k^{\alpha_{1+\sigma_{k-1}},...,\alpha_{\sigma_k}}\left(E^{*,k}_{\omega_1,...,\omega_n}\right)$$
with the right term being equal to $+\infty$ if $\omega_j=0$. This is absurd since $L_k^{\alpha_{1+\sigma_{k-1}},...,\alpha_{\sigma_k}}$ is a Lyapunov functional on the complete orbit $\gamma(z)$ by using Lemma \ref{Lemma:L_k-Lyap} whence we should have 
$$L_k^{\alpha_{1+\sigma_{k-1}},...,\alpha_{\sigma_k}}\left(E^{*,k}_{\tau_1,...,\tau_n}\right)\geq L_k^{\alpha_{1+\sigma_{k-1}},...,\alpha_{\sigma_k}}\left(E^{*,k}_{\omega_1,...,\omega_n}\right).$$
It follows that $\J_>=\emptyset$ whence $\tau_j\leq \omega_j$ for each $j\in \J$. If $\J_<\neq \emptyset$ then we would have
$$1=\sum_{j\in \J}\tau_j<\sum_{j\in \J}\omega_j=1$$
which is absurd. Consequently $\J_<=\emptyset$ and then $\tau_j=\omega_j$ for each $j\in\J$ which amounts to \eqref{Eq:Omega=tau}.

\underline{Step 5.} Finally we end up with $\gamma(z)=\{\phi(t), t\in \R\}$ being a homoclinic orbit through $z\in K$ with
\begin{equation*}
\lim_{t\to -\infty}\phi(t)=E^{*,k}_{\tau_1,...,\tau_n}=\lim_{t\to +\infty}\phi(t).
\end{equation*}
As in Step 2, let $(\alpha_{1+\sigma_{k-1}},...,\alpha_{\sigma_k})\in[0,1]^{\sigma_k-\sigma_{k-1}}$ such that $\sum_{j=1+\sigma_{k-1}}^{\sigma_k}\alpha_j=1$, with $\alpha_j=0$ for every $j\in\llbracket 1+\sigma_{k-1},\sigma_k\rrbracket \setminus \J$ and with $\alpha_j>0$ for every $j\in\J$. From Lemma \ref{Lemma:L_k-Lyap} we know that $L_k^{\alpha_{1+\sigma_{k-1}},...,\alpha_{\sigma_k}}$ is a Lyapunov functional on $\gamma(z)$ that is an homoclinic orbit (see the latter equation). It implies that $L_k^{\alpha_{1+\sigma_{k-1}},...,\alpha_{\sigma_k}}$ is constant on $\gamma(z)$. Again by Lemma \ref{Lemma:L_k-Lyap} we deduce that $\gamma(z)\subset \E_{k,\J}$ and then $z\in \E_{k,\J}$. Hence $K\subset \E_k$ and this latter set is isolated.

\item We prove that the set
$$\left\{\{E_0\}, \E_2, ..., \E_{n_>}\right\}$$
is acyclic (see \cite[Definition 8.14, p. 187]{SmithThieme2011} for the definition). Let $z\in \partial \S_0$ and let $\gamma(z)=\{\phi(t),t\in\R\}\subset \partial \S_0$ be a complete orbit through $z$.

$\textbf{(i)}$ Suppose that $\lim_{t\to \infty}\phi(t)=E_0$. Since $\gamma(z)\subset \partial \S_0$, then by looking at the reduced problem \eqref{Eq:model} with the infected populations $k\in\llbracket 2,n_>\rrbracket$ we deduce from the recurrence hypothesis that necessarily $z\in \cap_{j=1}^{\sigma_{n_>}}\partial \S_j$ (otherwise we would have $\lim_{t\to +\infty}d(\phi(t),\E_j)=0$ for $j\in\llbracket 2,n_>\rrbracket$ which is absurd). Consequently $\gamma(z)\subset \cap_{j=1}^{\sigma_{n_>}}\partial \S_j$. By Lemma \ref{Lemma:L0} we know that $L_0$ defined by \eqref{Eq:L0} is a Lyapunov functional on $\gamma(z)$. It follows from \cite[Proposition 2.51, p. 53]{SmithThieme2011} that $L_0$ is constant on $\alpha(\phi)$ and on $\omega(z)$. From the arguments above (see the case $\RR_{0,1}\leq 1$) we deduce that $\alpha(\phi)=\{E_0\}$ and that $L_0$ is constant on $\gamma(z)$. By Lemma \ref{Lemma:L0-constant} we deduce that $\gamma(z)=\{E_0\}$.

$\textbf{(ii)}$ Suppose that $n_>\geq 2$ and that there exists $k\in\llbracket 2, n_>\rrbracket$ such that $\lim_{t\to \infty}d(\phi(t),\E_k)=0$. Again, by recurrence hypothesis we necessarily have $z\in (\cap_{j=1}^{\sigma_{k-1}}\partial \S_j)\cap(\cup_{j=1+\sigma_{k-1}}^{\sigma_k}\S_j)$ (otherwise we would have $\lim_{t\to+\infty}d(\phi(t),\E_j)=0$ for $j\in\llbracket 2,k-1\rrbracket$ which is absurd). Following the proof of the point 3.(ii) (see Steps 2--5) by first defining the subset $\J$, then proving \eqref{Eq:Phi_limit} and \eqref{Eq:Omega=tau} we arrive at $\gamma(z)$ being a homoclinic orbit through $z\in \E_{k,\J}\subset \E_k$. We just proved that the set  $\left\{\{E_0\}, \E_2, ..., \E_{n_>}\right\}$ is not cyclic and consequently acyclic.
\end{enumerate}

Combining the four previous points with the existence of a global attractor $\A\subset X_+$ (by Corollary \ref{Cor:attractor}) imply from \cite[Theorem 8.17, p. 188]{SmithThieme2011} that $\Phi$ is uniformly weakly $\p$-persistent. Actually, from \cite[Theorem 5.2, p. 126]{SmithThieme2011} we know that the semiflow is uniformly strongly $\p$-persistent. We note here that we could have used \cite[Theorem 4.2]{HaleWaltman89} since $\Phi$ is bounded dissipative by using Corollary \ref{Cor:attractor}. Finally from \cite[Theorem 5.19, p. 138]{SmithThieme2011} (which can be used here since $\A$ is a strong global attractor) we deduce the existence of the global attractor $\A_0$ in $\cup_{k=1}^{\sigma_1}\S_k$ which is LAS.

\subsubsection{Lyapunov functional}
\label{Sec:Lyap-functional}

Let $\J\subset \llbracket 1,\sigma_1\rrbracket$ and $K$ be a compact invariant such that $K\subset (\cap_{k \in \J}\S_k)\cap (\cap_{k\in\llbracket 1,\sigma_1\rrbracket \setminus \J}\partial \S_k)$. Let $(\alpha_1,...,\alpha_n)\in[0,1]^n$ such that $E^{*,1}_{\alpha_1,...,\alpha_n}\in \E_{1,\J}$. We prove here that the functional $L_1^{\alpha_{1},...,\alpha_{\sigma_1}}$ defined by \eqref{Eq:Lk} is a Lyapunov functional on every complete orbit $\gamma(z)$ for any $z\in K$.

\begin{enumerate}
    \item We first show that $L_1^{\alpha_{1},...,\alpha_{\sigma_1}}(\Phi_t(z))$ is well-defined for every $(t,z)\in\R_+\times K$. Using \eqref{Eq:c_s} we have $\Phi_t^S(z)\in[c_S,\frac{\Lambda}{\mu_0}]$ for every $t\geq 0$. In \cite[Proposition 4.4.4]{Richard2020}, it was proved that for every $(z,k)\in K\times \J$, there exists $\tau\geq 0$ such that 
    $$\int_0^\infty \beta_k(a)\Phi_t^k(z)(a)da>0, \quad \forall t\geq \tau.$$
    We now proceed in three steps.
    
  \underline{Step 1.} We prove, as it is claimed in the proof of \cite[Proposition 5.1]{Richard2020} (but without proof) or in \cite[Proposition 2.3]{Perasso2019}, that $\tau$ is uniform in $z$, \textit{i.e.} for every $k\in \J$ then
\begin{equation}
\label{Eq:tau_k}
\exists \ \tau_k\geq 0, \quad \forall z\in K: \quad \int_0^\infty \beta_k(a)\Phi_t^k(z)(a)da>0, \quad \forall t\geq \tau_k.
\end{equation}
To show this, let $k\in \J$ and first assume that $\overline{\beta_k}<+\infty$. Defining $ \overline{c}=(\overline{\beta_k}+\underline{\beta_k})/2$ then:
\begin{enumerate}[label=\textbf{(\roman*)}]
    \item There exists $\delta\in(0,\overline{c}-\underline{\beta_k})$ such that $\int_0^{\overline{\beta_k}-\delta}\Phi_0^k(z)(a)da>0$ for every $z\in K$. Indeed, otherwise by contradiction one gets a sequence $(\delta_n,z_n)\subset ((0,c-\underline{\beta_k}) \times K)^{\N}$ such that $\lim_{n\to \infty}\delta_n=0$ and $\int_0^{\overline{\beta_k}-\delta_n}\Phi_0^k(z_n)(a)da=0$ for each $n\in\N$. By compactness of $K$ there exists (up to a subsequence) $\overline{z}\in K$ such that $\lim_{n\to \infty}\|z_n-\overline{z}\|_{X}=0$. But then
    \begin{flalign*}
    \left|\int_0^{\overline{\beta_k}}\Phi_0^k(\overline{z})(a)da\right|&\leq \int_0^{\overline{\beta_k}-\delta_n}\left|\Phi_0^k(\overline{z})(a)-\Phi_0^k(z_n)(a)\right|da+\int_0^{\overline{\beta_k}-\delta_n}\Phi_0^k(z_n)(a)da+\int_{\overline{\beta_k}-\delta_n}^{\overline{\beta_k}}\Phi_0^k(\overline{z})(a)da\\
    &\leq \|\overline{z}-z_n\|_{X}+\|\Phi_0^k(\overline{z})\chi_{[\overline{\beta_k}-\delta_n,\overline{\beta_k}]}\|_{L^1}\underset{n\to \infty}{\longrightarrow} 0
    \end{flalign*}
    by the dominated convergence theorem, whence $\int_0^{\overline{\beta_k}}\Phi_0^k(\overline{z})(a)da=0$ which is absurd since $\overline{z}\in K\subset \S_k$. 
    \item For every $z\in K$, there exists $\overline{a}\in[0,\overline{\beta_k}-2\delta]$ such that $\int_{\overline{a}}^{\overline{a}+\delta}\Phi_0^k(z)(a)da>0$ (otherwise there would be $\overline{z}\in K$ such that $\int_0^{\overline{\beta_k}-\delta}\Phi_0^k(\overline{z})(a)da=0$ which contradicts the first item).
    \item For every $z\in K$, there exists $t_0\in[0,\delta+\underline{\beta_k}]$ such that $\int_{\underline{\beta_k}}^{\underline{\beta_k}+\delta}\Phi_{t_0}^k(z)(a)da>0$. Indeed, if $\overline{a}\leq \underline{\beta_k}$ then $t_0:=\underline{\beta_k}-\overline{a}$ works because
    $$\int_{\underline{\beta_k}}^{\underline{\beta_k}+\delta}\Phi_{t_0}^k(z)(a)da\geq e^{-\mu_0 t_0}\int_{\overline{a}}^{\overline{a}+\delta}\Phi_0^k(z)(a)da>0$$
    by the second item. Otherwise (if $\overline{a}>\underline{\beta_k}$) we can take $t_0:=\delta+\underline{\beta_k}$ and we get
    \begin{flalign*}
    \int_{\underline{\beta_k}}^{\underline{\beta_k}+\delta} \Phi_{t_0}^k(z)(a)da&\geq e^{-\mu_0 \underline{\beta_k}}\int_{\underline{\beta_k}}^{\underline{\beta_k}+\delta}\Phi_{t_0-\underline{\beta_k}}^k(z)(a-\underline{\beta_k})da \\
    &\geq e^{-\mu_0 \underline{\beta_k}}\int_0^\delta \Phi_{\delta}^k(z)(a)da\\
    &\geq e^{-\mu_0(\underline{\beta_k}+\delta)}\int_0^\delta \Phi_{\delta-a}^k(z)(0)da\\
    &\geq e^{-\mu_0(\underline{\beta_k}+\delta)}\int_0^\delta \Phi_{\delta-a}^S(z)\int_{\overline{a}+\delta-a}^{\overline{a}+2\delta-a}\beta_k(\xi)\Phi_{\delta-a}^k(z)(\xi)d\xi da \\
    &\geq c_S e^{-\mu_0(\underline{\beta_k}+\delta)}\int_0^\delta e^{-\mu_0 (\delta-a)}\int_{\overline{a}+\delta-a}^{\overline{a}+2\delta-a} \beta_k(\xi)\Phi_0^k(z)(\xi-\delta+a)d\xi da \\
    &\geq c_S e^{-\mu_0(\underline{\beta_k}+\delta)}\int_0^\delta e^{-\mu_0 (\delta-a)}\int_{\overline{a}}^{\overline{a}+\delta} \beta_k(\xi+\delta-a)\Phi_0^k(z)(\xi)d\xi da>0 
    \end{flalign*}
    due to the second item and the fact that $\xi+\delta-a\in[\overline{a},\overline{a}+2\delta]\subset[\underline{\beta_k},\overline{\beta_k}]$.
    \item We observe that for every $t\in[t_0,t_0+\delta]$, we have:
    \begin{flalign*}
    \Phi_t^k(z)(0)&\geq c_S\int_{\underline{\beta_k}+t-t_0}^{\underline{\beta_k}+\delta+t-t_0}\beta_k(a)\Phi_t^k(z)(a)da  \\
    &\geq c_S\int_{\underline{\beta_k}+t-t_0}^{\underline{\beta_k}+\delta+t-t_0}\beta_k(a)\Phi_{t_0}^k(z)(a-(t-t_0))da \\
    &\geq c_S\int_{\underline{\beta_k}}^{\underline{\beta_k}+\delta}\beta_k(a+t-t_0)\Phi_{t_0}^k(z)(a)da>0
    \end{flalign*}
    since $a+t-t_0\in[\underline{\beta_k},\underline{\beta_k}+2\delta)\subset (\underline{\beta_k},\overline{\beta_k})$ by definition of $\delta$ and we deduce that $\Phi_t^k(z)(0)>0$ for every $t\in[t_0,t_0+\delta]$.
    \item Similarly, for every $t\in(t_0+\underline{\beta_k},t_0+\underline{\beta_k}+2\delta)$, we see that:
    \begin{flalign*}
    \Phi_t^k(z)(0)&\geq c_S\int_{\underline{\beta_k}}^{\underline{\beta_k}+\delta}\beta_k(a)\Phi_t^k(z)(a)da \\ 
    &\geq c_S\int_{\underline{\beta_k}}^{\underline{\beta_k}+\delta}\beta_k(a)\Phi_{t-a}^k(z)(0)e^{-\mu_0 a}\chi_{[0,t]}(a)da\\
    &\geq c_S\int_{t-(\underline{\beta_k}+\delta)}^{t-\underline{\beta_k}}\beta_k(t-a)\Phi_a^k(z)(0)e^{-\mu_0(t-a)}da \\
    &\geq c_S\int_{\max\{t-(\underline{\beta_x}+\delta),t_0\}}^{\min\{t_0+\delta,t-\underline{\beta_k}\}}\beta_k(t-s)\Phi_s^k(z)(0)e^{-\mu_0(t-s)}ds>0 \\
    \end{flalign*}
    for every $t\in(t_0+\underline{\beta_k},t_0+\underline{\beta_k}+2\delta)$ since $\Phi_s^k(z)(0)>0$ for every $s\in[t_0,t_0+\delta]$ and $t-a\in[\underline{\beta_k},\underline{\beta_k}+\delta]$.
    \item The same computations prove that for every $n\geq 0$ and every $t\in(t_0+n\underline{\beta_k},t_0+n\underline{\beta_k}+(n+1)\delta)$ we have $\Phi_t^k(z)(0)>0$. Let $N=\left\lceil \underline{\beta_k}/\delta \right\rceil$ so that $N\delta\geq \underline{\beta_k}$, then it follows that
    $$\Phi_t^k(z)(0)>0 \quad \forall t>t_0+N\underline{\beta_k}.$$
    \item In conclusion, we proved that for every $z\in K$, there exists $t_0\in[0,\delta+\underline{\beta_k}]$ such that $\Phi_t^k(z)(0)>0$ for every $t>t_0+N\underline{\beta_k}$ with $\delta\in(0,c-\underline{\beta_k})$ and $N=\left\lceil \underline{\beta_k}/\delta \right\rceil$. It then leads to 
    $$\forall z\in K, \quad \Phi_t^k(z)(0)>0, \quad \forall t>\delta+(N+1)\underline{\beta_k}$$
    which implies \eqref{Eq:tau_k} (here $\tau_k=\delta+(N+1)\underline{\beta_k}$ is independent of $z$).
    \item The case $\overline{\beta_k}=\infty$ is treated similarly. We just have to let $\delta=\underline{\beta_k}$. We prove as (ii) that for every $z\in K$, there exists $\overline{a}\leq \infty$ such that $\int_{\overline{a}}^{\overline{a}+\delta}\Phi_0^k(z)(a)da>0$. Then as in (iii), we prove that for every $z\in K$, there exists $t_0\in[0,\delta+\underline{\beta_k}]=[0,2\underline{\beta_k}]$ such that $\int_{\underline{\beta_k}}^{2\underline{\beta_k}}\Phi_{t_0}^k(z)(a)da>0$ (in the case $\overline{a}\leq \underline{\beta_k}$ we may consider $t_0:=\underline{\beta_k}-\overline{a}$ and in the case $\overline{a}>\underline{\beta_k}$ we may consider $t_0=2\underline{\beta_k}$). We also prove (iv), \textit{i.e.} $\Phi_t^k(z)(0)>0$ for every $t\in[t_0,t_0+\delta]$. The step (v) is similar, while in step (vi), we have $N=\left\lceil \underline{\beta_x}/\delta \right\rceil=1$ so that
    $$\Phi_t^k(z)(0)>0, \quad \forall t>t_0+\underline{\beta_k}.$$
    Consequently, \eqref{Eq:tau_k} is still true with $\tau_k>3\underline{\beta_k}$.
\end{enumerate} 

    \underline{Step 2.} By invariance of $K$ we deduce that for every $(z,k)\in K\times \J$, there exists a complete orbit $\{\phi(t),t\in \R\}\subset K$ through $z$ leading to
    $$\int_0^\infty \beta_k(a)\Phi_0^k(z)(a)da=\int_0^\infty \beta_k(a)\Phi_0^k(\phi(0))(a)da=\int_0^\infty \beta_k(a)\Phi_{\overline{t}}^k(\phi(-\overline{t}))(a)da>0$$
    by using \eqref{Eq:tau_k}. Since $K$ is compact then a continuity argument implies that
    \begin{equation*}
    \forall k\in\J, \exists \ c_k>0, \ \forall z\in K: \quad \int_0^\infty \beta_k(a)\Phi_0^k(z)(a)da\geq c_k.
    \end{equation*}
    Again, by invariance of $K$, we deduce that for every $(z,k)\in K\times \J$, there exists a complete orbit $\{\phi(t),t\in \R\}\subset K$ through $z$ leading for every $t\geq 0$ to
    $$\int_0^\infty \beta_k(a)\Phi_t^k(z)(a)da=\int_0^\infty \beta_k(a)\Phi_t^k(\phi(0))(a)da=\int_0^\infty \beta_k(a)\Phi_0^k(\phi(t))(a)da\geq c_k$$
    that is to say
    \begin{equation}\label{Eq:int_c}
    \forall k\in \J, \exists \ c_k>0, \forall t\geq 0, \ \forall z\in K: \int_0^\infty \beta_k(a)\Phi_t^k(z)(a)da\geq c_k.
    \end{equation}
    Now we claim that 
    \begin{equation}\label{Eq:cs_cx}
    \forall (t,a,z,k)\in (\R_+)^2\times K\times \J: \quad \Phi_t^k(z)(a)\geq c_S c_k \pi_k(a)
    \end{equation}
    where we remind that $c_S$ is defined in \eqref{Eq:c_s}. Let $k\in \J$. The proof is similar to \cite[Proposition 5.1.2]{Richard2020} but expressed on $K$ rather than on the omega-limit set of some initial condition of $\S_k$, hence we give its proof for completeness. Let $(t,a)\in (\R_+)^2$ such that $t>a$. From the expression of the semiflow \eqref{Eq:Duhamel_phik} combined with \eqref{Eq:c_s} and \eqref{Eq:int_c} we get \eqref{Eq:cs_cx} for every $t>a$. Suppose now that $t\leq a$. Once again by invariance of $K$, we know that for every $z\in K$, there exists a complete orbit $\{\phi(t),t\in \R\}\subset K$ through $z$ leading to
    $$\Phi_t^k(z)(a)=\Phi_t^k(\phi(0))(a)=\Phi_{a+1}^k(\phi(t-(a+1)))(a)\geq c_Sc_k\pi_k(a)$$
    and \eqref{Eq:cs_cx} is proved.
    
    \underline{Step 3.} Let $\J^*=\{k\in\J: \alpha_k>0\}$. We prove now that 
    \begin{equation}
    \label{Eq:well-def_Lk}
    \exists \ c_{\Phi}>0, \ \forall (t,a,z,k)\in(\R_+)^2\times K\times \J^*: \quad \left(\dfrac{x_{k,\alpha_k}^*(a)}{\Phi_t^k(z)(a)}-1\right)^2\leq c_{\Phi}.
    \end{equation}
    Let $(k,t,a)\in \J^*\times(\R_+)^2$. We deduce from \eqref{Eq:cs_cx} that
    \begin{equation*}
    \forall z\in K: \quad \dfrac{\Phi_t^k(z)(a)}{x_{k,\alpha_k}^*(a)}\geq \dfrac{c_Sc_k r_k}{\mu_S \alpha_k (R_0^k-1)}=:\tilde{c}_k>0
    \end{equation*}
    leading to
    $$\left(\dfrac{x_{k,\alpha_k}^*(a)}{\Phi_t^k(z)(a)}-1\right)^2\leq \dfrac{1}{(\tilde{c}_k)^2}+1<\infty$$
    which proves \eqref{Eq:well-def_Lk} when considering
    $$c_{\Phi}=\max_{k\in \J^*}\left\{\dfrac{1}{\tilde{c}^2_k}+1
    \right\}.$$

    \underline{Step 4.} Let $(t,a,z,k)\in (\R_+)^2\times K\times \J^*$, then from \eqref{Eq:well-def_Lk} we get
    \begin{flalign}
        x_{k,\alpha_k}^*(a)g\left(\dfrac{\Phi_t^k(z)(a)}{x_{k,\alpha_k}^*(a)}\right)&=x_{k,\alpha_k}^*(a)\left(\dfrac{\Phi_t^k(z)(a)}{x_{k,\alpha_k}^*(a)}+\ln\left(\dfrac{x_{k,\alpha_k}^*(a)}{\Phi_t^k(z)(a)}\right)-1\right) \nonumber \\
        &\leq  x_{k,\alpha_k}^*(a)\left(\dfrac{\Phi_t^k(z)(a)}{x_{k,\alpha_k}^*(a)}+\dfrac{x_{k,\alpha_k}^*(a)}{\Phi_t^k(z)(a)}-2\right)  =\Phi_t^k(z)(a)\left(\dfrac{x_{k,\alpha_k}^*(a)}{\Phi_t^k(z)(a)}-1\right)^2 \nonumber \\
        &\leq c_{\Phi}\Phi_t^k(z)(a)
    \label{Eq:g_xstar}
    \end{flalign}
    by using the fact that $\ln(x)\leq x-1$ for each $x>0$.
    
    \underline{Step 5.} Finally, from \eqref{Eq:g_xstar} and using the integrability on $\R_+$ of the functions 
    $$a\longmapsto \Psi_k(a)\Phi_t^k(z)(a), \quad \forall (t,z)\in \R_+\times K$$
    for each $k\in \llbracket 1,n\rrbracket$, it follows that $(t,z)\longmapsto L_1^{\alpha_1,...,\alpha_{\sigma_1}}(\Phi_t(z))$ is well-defined on $\R_+\times K$.
    
    \item We prove that $L^{\alpha_1,...,\alpha_{\sigma_1}}_1$ is a Lyapunov function on every complete orbit $\gamma(z)$ for any $z\in K$. This amounts to showing that $L^{\alpha_1,...,\alpha_{\sigma_1}}_1$ is a Lyapunov function on each positive orbit $\gamma^+(z)$ for every $z\in K$. Let $z:=(S_0,x_{1,0},...,x_{n,0})\in K$ and let $\widehat{z}:=(S_0,(0,x_{1,0}),...,(0,x_{n,0}))$. We will follow the sketch of proof of Lemma \ref{Lemma:L0} (for the case $\RR_{0,1}\leq 1$), which is here more tedious because of the shape of the functional.
    
    $\textbf{(i)}$ Suppose first that $\widehat{z}\in D((A+F)_0)\cap \X_{0+}$ where $D((A+F)_0)$ is defined by \eqref{Eq:domain_A+F}. Then using \cite[Theorem 5.6.6, p. 242]{MagalRuan2018} we can compute the derivative of the function $t\longmapsto L_1^{\alpha_1,...,\alpha_{\sigma_1}}(\Phi_t(z))$ according to $t$ (which is well-defined due to the first point) since $\Phi_t^k(z)$ belong to $W^{1,1}(0,\infty)$ for each $t\geq 0$ and every $k\in\llbracket 1,n\rrbracket$. Similar computations to the case $\RR_{0,1}\leq 1$ (see \eqref{Eq:Lk_deriv-inf}) lead to
    \begin{flalign}
    \label{Eq:L1_deriv-inf}
       &\dfrac{dL_1^{\alpha_{1},...,\alpha_{\sigma_1}}(\Phi_{t}(z))}{dt}=-\displaystyle\sum_{k\in\llbracket 1+\sigma_1,n\rrbracket\cup(\J\setminus \J^*)} \left(\dfrac{1}{r_k}-S^*_{\sigma_1}\right)\int_0^\infty \beta_k(a)\Phi_t^k(z)(a)da-\dfrac{\mu_S}{\Phi_t^S(z)}(\Phi_t^S(z)-S^*_{\sigma_1})^2 \nonumber \\
       &\qquad -|\J^*|g\left(\dfrac{S^*_{\sigma_1}}{\Phi_t^S(z)}\right)-\displaystyle\sum_{k\in\J^*} S^*_{\sigma_1} \int_0^\infty \beta_k(a)x_{k,\alpha_k}^*(a)g\left(\dfrac{\Phi_t^k(z)(a)\int_0^\infty \beta_k(s)x_{k,\alpha_k}^*(s)ds}{x_{k,\alpha_k}^*(a)\int_0^\infty \beta_k(s)\Phi_t^k(z)(s)ds}\right)da\leq 0
    \end{flalign}
    since $\frac{1}{r_k}>S^*_{\sigma_1}=\frac{1}{r_{\sigma_1}}$ for every $k\in\llbracket 1+\sigma_1,n\rrbracket$ and $\frac{1}{r_k}=\frac{1}{r_{\sigma_1}}$ for every $k\in\llbracket 1,\sigma_1\rrbracket$. It can be rewritten in the following form:
    \begin{flalign}
    \label{Eq:Lk_int}
    L_1^{\alpha_1,...,\alpha_{\sigma_1}}(\Phi_{t}(z))&=L_1^{\alpha_1,...,\alpha_{\sigma_1}}(z)-\displaystyle\sum_{k=1+\sigma_1}^n  \left(\dfrac{1}{r_k}-S^*_{\sigma_1}\right)\int_0^t \int_0^\infty \beta_k(a)\Phi_\xi^k(z)(a)da d\xi \nonumber \\
    & -\int_0^t \dfrac{\mu_S}{\Phi_\xi^S(z)}(\Phi_\xi^S(z)-S^*_{\sigma_1})^2 d\xi- |\J^*|\displaystyle\int_0^t g\left(\dfrac{S_{\sigma_1}^*}{\Phi_{\xi}^S(z)}\right)d\xi \nonumber \\
    & - \sum_{k\in\J^*} S^*_{\sigma_1}\int_0^t \int_0^\infty \beta_k(a)x_{k,\alpha_k}^*(a)g\left(\dfrac{\Phi_\xi^k(z)(a)\int_0^\infty \beta_k(s)x_{k,\alpha_k}^*(s)ds}{x_{k,\alpha_k}^*(a)\int_0^\infty \beta_k(s)\Phi_\xi^k(z)(s)ds}\right)da d\xi
    \end{flalign}

    $\textbf{(ii)}$ Now, suppose that $\widehat{z}\not\in D((A+F)_0)\cap \X_{0+}$. We will use a density argument to compute the derivative of the functional $L_1^{\alpha_1,...,\alpha_{\sigma_1}}$ along the solution of \eqref{Eq:model} for the initial condition $z$. However, taking a sequence of initial conditions in $D((A+F)_0)\cap \X_{0+}$ is not sufficient for the functional $L_1^{\alpha_1,...,\alpha_{\sigma_1}}$ to be well-defined because similar estimates to \eqref{Eq:cs_cx} are necessary for the sequence of initial conditions. Hence we will build a suitable sequence.
    
    \textbf{Step 1.} Let $t\geq 0$ and $j\in \J$. We know by \eqref{Eq:cs_cx} that $\Phi_s^k(z)(a)\geq c_Sc_k\pi_k(a)$ for every $(s,a,z,k)\in (\R_+)^2\times K\times \J$. We define
    $$v=\left(\tilde{S}_0,\tilde{x}_{1,0},...,\tilde{x}_{n,0}\right), \qquad \widehat{v}=\left(\tilde{S}_0,(0,\tilde{x}_{1,0}),...,(0,\tilde{x}_{n,0})\right)$$
    where
    $$\tilde{S}_0=\dfrac{S_0\int_0^\infty \beta_j(a)x_{j,0}(a)da-\nu}{\int_0^\infty \beta_j(a)x_{j,0}(a)da-\frac{\nu}{S_j^*}}$$
    and
    $$\tilde{x}_{k,0}=\begin{cases}
    x_{k,0}-\dfrac{\nu\pi_k \int_0^\infty \beta_k(a)x_{k,0}(a)da}{\int_0^\infty \beta_j(a)x_{j,0}(a)da} & \text{ if } k\in \J \\
    0 & \text{ if } k\in \llbracket 1,n\rrbracket \setminus \J
    \end{cases}$$
    with $\nu>0$ small enough such that $v\in X_+$ and $\widehat{v}\in \X_{0+}$ (which exist by using \eqref{Eq:cs_cx}).
    Here we use the density of $D((A+F)_0)\cap \X_{0+}$ into $\X_{0+}$ to assert:
$$\exists \ (\widehat{v}^{(m)})_{m\in\N}\subset \left(D((A+F)_0)\cap \X_{0+}\right)^{\N}:  \|\widehat{v}^{(m)}-\widehat{v}\|_{\X}\underset{m\to \infty}{\longrightarrow} 0$$
where $\widehat{v}^{(m)}=(S_{0}^{(m)},(0,x_{1,0}^{(m)}),...,(0,x_{n,0}^{(m)}))\in D((A+F)_0)\cap \X_{0+}$. Since $\Phi$ is state-continuous uniformly in finite time (see Proposition \ref{Prop:solutions} 4.), then the sequence $v^{(m)}=(S_{0}^{(m)},x_{1,0}^{(m)},...,x_{n,0}^{(m)})\in X_{+}$ satisfies
$$\sup_{s\in[0,t]}\|\Phi_s(v^{(m)})-\Phi_s(v)\|_{X}\underset{m\to \infty}{\longrightarrow}0.$$

Now we define the sequence $(\widehat{w}^{(m)})_{m\in \N}:=(\overline{S}^{(m)}_{0},(0,\overline{x}^{(m)}_{1,0}),...,(0,\overline{x}^{(m)}_{n,0}))_{m\in \N}\in (\X_{0+})^{\N}$
where
$$\overline{S}^{(m)}_{0}=\dfrac{S^{(m)}_{0}\int_0^\infty \beta_j(a) x^{(m)}_{j,0}(a)da+\nu}{\int_0^\infty \beta_j(a) x^{(m)}_{j,0}(a)da+\frac{\nu}{S_j^*}}$$
and
$$\overline{x}^{(m)}_{k,0}=\begin{cases}
x^{(m)}_{k,0}+\dfrac{\nu\pi_k \int_0^\infty \beta_k(a)x^{(m)}_{k,0}(a)da}{\int_0^\infty \beta_j(a)x^{(m)}_{j,0}(a)da} & \text{ if } k\in \J \\
0 & \text{ if } k\in\llbracket 1,n\rrbracket \setminus \J
\end{cases}$$
which is well-defined since we have $\int_0^\infty \beta_j(a) x^{(m)}_{j,0}(a)da\underset{m\to \infty}{\longrightarrow} \int_0^\infty \beta_j(a) x_{j,0}(a)da-\frac{\nu}{S_j^*}>0$ by definition of $\nu$. It follows that $\int_0^\infty \beta_j(a)x^{(m)}_{j,0}(a)da>0$ for each $m\geq 0$ (up to a sub-sequence). Recalling that $x_{k,0}\in \partial \S_k$ for every $k\in\llbracket 1,n\rrbracket \setminus \J$ (by definition of $K$) and by invariance of $K$, then $x_{k,0}(a)=0$ f.a.e. $a\geq 0$. By construction we then see that $\lim_{m\to \infty}\|\widehat{w}^{(m)}-\widehat{z}\|_{\X}=0$ and $\lim_{m\to \infty}\|w^{(m)}-z\|_{X}=0$ where 
$$w^{(m)}:=(\overline{S}^{(m)}_{0},\overline{x}_{1,0}^{(m)},...,\overline{x}_{n,0}^{(m)})\in X_+$$
leading to
\begin{equation}\label{Eq:sup_phi0}
\sup_{s\in[0,t]}\|\Phi_s(w^{(m)})-\Phi_s(z)\|_{X}\underset{m\to \infty}{\longrightarrow}0.
\end{equation}

We now check that $\widehat{w}^{(m)}\in D((A+F)_0)$. Since $\widehat{w}^{(m)}\in D(A)$, it remains to prove that $A\widehat{w}^{(m)}+F(\widehat{w}^{(m)})\in \X_0$ \textit{i.e.} 
\begin{equation}\label{Eq:domain_wn}
\forall k\in \J, \qquad \overline{x}^{(m)}_{k,0}(0)=\overline{S}^{(m)}_{0}\int_0^\infty \beta_k(a) \overline{x}^{(m)}_{k,0}(a)da.
\end{equation}
Knowing that $\widehat{v}^{(m)}\in D((A+F)_0)$ then by construction we have for every $k\in \J$:
    \begin{flalign*}
    \overline{x}^{(m)}_{k,0}(0)&=x^{(m)}_{k,0}(0)+\dfrac{\nu\int_0^\infty \beta_k(a)x^{(m)}_{k,0}(a)da}{\int_0^\infty \beta_j(a)x_{j,0}^{(m)}(a)da}=\dfrac{\int_0^\infty \beta_k(a)x^{(m)}_{k,0}(a)da}{\int_0^\infty \beta_j(a)x_{j,0}^{(m)}(a)da}\left(S^{(m)}_{0}\int_0^\infty \beta_j(a)x_{j,0}^{(m)}(a)da+\nu\right) \\
    &=\dfrac{\overline{S}^{(m)}_{0}\int_0^\infty \beta_k(a)x^{(m)}_{k,0}(a)da}{\int_0^\infty \beta_j(a)x_{j,0}^{(m)}(a)da}\left(\int_0^\infty \beta_j(a)x_{j,0}^{(m)}(a)da+\dfrac{\nu}{S_j^*}\right) \\
    &=\overline{S}_{0}^{(m)}\left(\int_0^\infty \beta_k(a)x^{(m)}_{k,0}(a)da+\left(\int_0^\infty \beta_k(a)\pi_k(a)da\right)\dfrac{\nu\int_0^\infty \beta_j(a)x^{(m)}_{j,0}(a)da}{\int_0^\infty \beta_j(a)x^{(m)}_{j,0}(a)da}\right)\\
    &=\overline{S}_{0}^{(m)}\int_0^\infty \beta_k(a)\overline{x}^{(m)}_{k,0}(a)da
    \end{flalign*}
    so \eqref{Eq:domain_wn} is satisfied and $(\widehat{w}^{(m)})_{m\in \N}\subset (D((A+F)_0)\cap \X_{0+})^{\N}$. 
    
    \textbf{Step 2.} Let $\ep>0$ small enough such that
    \begin{equation}\label{Eq:eps_small}
    \forall k\in \J: \quad c_k-\ep\|\beta_k\|_{L^\infty}>0 \quad \text{and} \quad c_S-\ep>0
    \end{equation}
    holds. We show first a property analogous to \eqref{Eq:cs_cx} which is
    \begin{flalign}\label{Eq:cw}
    \exists \ c_w>0, \exists M\in \N, \forall (s,a,k)\in[0,t]\times \R_+\times \J, \forall m\geq M: \nonumber  \\
    \Phi_s^k(w^{(m)})(a)\geq c_w \pi_k(a) \text{ and } \sup_{s\in[0,t]}\|\Phi_s(w^{(m)})-\Phi_s(z)\|_{X}\leq \ep.
    \end{flalign}
    Let $k\in \J$ then
    $$\dfrac{\int_0^\infty \beta_k(a)x_{k,0}^{(m)}(a)da}{\int_0^\infty \beta_j(a)x_{j,0}^{(m)}(a)da}\underset{m\to \infty}{\longrightarrow}\dfrac{\int_0^\infty \beta_k(a)x_{k,0}(a)da-\frac{\nu \int_0^\infty \beta_k(a)x_{k,0}(a)da}{S^*_{k}\int_0^\infty \beta_j(a)x_{j,0}(a)da}}{\int_0^\infty \beta_j(a)x_{j,0}(a)da-\frac{\nu}{S_j^*}}=\dfrac{\int_0^\infty \beta_k(a)x_{k,0}(a)da}{\int_0^\infty \beta_j(a)x_{j,0}(a)da}>0.$$
    By definition of $\overline{x}^{(m)}_{k,0}$ and by using \eqref{Eq:Duhamel_phik} we deduce that there exists $M_k\in \N$ and $c_{w,k}>0$ such that for every $m\geq M_k$ and every $(s,a)\in[0,t]\times [s,\infty)$ 
    \begin{flalign*}
    \Phi_s^{k}(w^{(m)})(a)&=\overline{x}^{(m)}_{k,0}(a-s)e^{-\int_{a-s}^a \mu_k(\xi)d\xi} \\
    &\geq \nu\pi_k(a-s)e^{-\int_{a-s}^a \mu_k(\xi)d\xi}\left(\dfrac{\int_0^\infty \beta_k(a)x_{k,0}^{(m)}(a)da}{\int_0^\infty \beta_j(a)x_{j,0}^{(m)}(a)da}\right)=\nu\pi_k(a)\left(\dfrac{\int_0^\infty \beta_k(a)x_{k,0}^{(m)}(a)da}{\int_0^\infty \beta_j(a)x_{j,0}^{(m)}(a)da}\right) \\
    &\geq c_{w,k}\pi_k(a).
    \end{flalign*}
    Considering $c_w^{(1)}=\min_{k\in \J}\{c_{w,k}\}>0$ and $M^{(1)}=\max_{k\in \J}\{M_k\}$ then $\Phi_s^k(w^{(m)})(a)\geq c_{\omega}^{(1)}\pi_k(a)$ for every $m\geq M^{(1)}$ and every $(s,a,k)\in[0,t]\times [s,\infty)\times \J$. From \eqref{Eq:sup_phi0} there exists $M^{(2)}\in \N$ such that $    \sup_{s\in[0,t]}\|\Phi_s(w^{(m)})-\Phi_s(z)\|_{X}\leq \ep$ for every $m\geq M^{(2)}$  and every $k\in \J$. Also, using \eqref{Eq:int_c}, we see that for each $k\in \J$ we have
    \begin{flalign}\label{Eq:int_c-wm}
        \int_0^\infty \beta_k(a)\Phi_s^k(w^{(m)})(a)da&=\int_0^\infty \beta_k(a)\Phi_s^k(z)(a)da+\int_0^\infty \beta_k(a)\left(\Phi_s^k(w^{(m)})(a)-\Phi_s^k(z)(a)\right)da \nonumber \\
        &\geq c_k-\ep\|\beta_k\|_{L^\infty}
    \end{flalign}
    for every $s\in[0,t]$ and every $m\geq M^{(2)}$. Moreover by \eqref{Eq:c_s} we have
    \begin{equation}\label{Eq:estimate_phi_S}
    \ep+\dfrac{\Lambda}{\mu_0}\geq \ep+\Phi_s^{S}(v)\geq \Phi_s^{S}(w^{(m)})\geq \Phi_s^{S}(v)-\ep\geq c_S-\ep, \quad \forall m\geq M^{(2)}
    \end{equation}
    and using \eqref{Eq:Duhamel_phik} we have
    $$\Phi_s^k(w^{(m)})(a)\geq (c_S-\ep)(c_k-\ep\|\beta_k\|_{L^\infty})\pi_k(a)=:c_{\ep,k}\pi_k(a)$$
    (wherein $c_{\ep,k}>0$ by \eqref{Eq:eps_small}) for every $(s,a,k)\in[0,t]\times [0,s)\times \J$ and every $m\geq M^{(2)}$. Considering $c_w^{(2)}=\min_{k\in \J}\{c_{\ep,k}\}$ implies that $\Phi_s^k(w^{(m)})(a)\geq c_{\omega}^{(2)}\pi_k(a)$ for every $m\geq M^{(2)}$ and every $(s,a,k)\in[0,t]\times [0,s)\times \J$. Finally, considering $c_w=\min\{c_w^{(1)}, c_w^{(2)}\}$ and $M=\max\{M^{(1)}, M^{(2)}\}$ prove \eqref{Eq:cw}. We then get
    $$\forall (s,a,k)\in[0,t]\times \R_+\times \J^*, \ \forall m\geq M: \quad \dfrac{\Phi_s^k(w^{(m)})(a)}{x_{k,\alpha_k}^*(a)}\geq \dfrac{c_wr_k}{\mu_S \alpha_k (R_0^k-1)}=:\overline{c}_k>0.$$
    Following the steps 3-4 of the first point, we deduce that
    $$x_{k,\alpha_k}^*(a)g\left(\dfrac{\Phi_s^k(w^{(m)})(a)}{x_{k,\alpha_k}^*(a)}\right)\leq \Phi_s^k(w^{(m)})(a)\left(\dfrac{1}{\overline{c}_k^2}+1\right)$$ 
    for every $(s,a,k)\in[0,t]\times \R_+\times \J^*$ and every $m\geq M$. Using \eqref{Eq:cw} we see that for each $k\in \J$ we have
    $$0\leq \int_0^\infty \Psi_k(a)\Phi_s^k(w^{(m)})(a)da\leq \ep\|\Psi_k\|_{L^{\infty}}+\int_0^\infty \Psi_k(a)\Phi_s^k(z)(a)da, \quad \forall s\in[0,t]$$
    which implies that $s\longmapsto L_1^{\alpha_1,...,\alpha_{\sigma_1}}(\Phi_s(w^{(m)}))$ is well-defined on $[0,t]$ for each $m\geq M$. It follows that we can compute its derivative as in the first point of this section (in the case where the initial condition is in the domain of the operator), leading for every $s\in[0,t]$ and every $m\geq M$ to:
    \begin{flalign}\label{Eq:Lk_int_m} 
    &L_1^{\alpha_1,...,\alpha_{\sigma_1}}(\Phi_{s}(w^{(m)}))=L_1^{\alpha_1,...,\alpha_{\sigma_1}}(w^{(m)})-\displaystyle\sum_{k=1+\sigma_1}^n \left(\dfrac{1}{r_k}-S^*_{\sigma_1}\right) \int_0^s \int_0^\infty \beta_k(a)\Phi_\xi^k(w^{(m)})(a)da d\xi \nonumber \\
    &\qquad-\displaystyle \int_0^s \dfrac{\mu_S}{\Phi_\xi^S(w^{(m)})}(\Phi_\xi^S(w^{(m)})-S^*_{\sigma_1})^2 d\xi- |\J^*|\int_0^s g\left(\dfrac{S^*_{\sigma_1}}{\Phi_\xi^S(w^{(m)})}\right)d\xi \nonumber \\
    &\qquad - \displaystyle\sum_{k\in\J^*}\int_0^s S_{\sigma_1}^* \int_0^\infty \beta_k(a)x_{k,\alpha_k}^*(a)g\left(\dfrac{\Phi_\xi^k(w^{(m)})(a)\int_0^\infty \beta_k(s)x_{k,\alpha_k}^*(s)ds}{x_{k,\alpha_k}^*(a)\int_0^\infty \beta_k(s)\Phi_\xi^k(w^{(m)})(s)ds}\right)da d\xi.
    \end{flalign}
    
    \textbf{Step 3.} It remains to show the formula \eqref{Eq:Lk_int_m} but for the initial condition $z$ instead of $w^{(m)}$. Let $\ep>0$ small enough such that \eqref{Eq:eps_small} is satisfied. Let $m\geq M$ (where $M$ is defined in \eqref{Eq:cw}). We thus compute for $s\in[0,t]$:
    \begin{equation}\label{Eq:L1_deriv-z}
    \begin{array}{rcl}
    &&\H_s(z):=\left|L_1^{\alpha_1,...,\alpha_{\sigma_1}}(\Phi_{s}(z))-L_1^{\alpha_1,...,\alpha_{\sigma_1}}(z)+\displaystyle\sum_{k=1+\sigma_1}^n \left(\dfrac{1}{r_k}-S^*_{\sigma_1}\right)\int_0^s \int_0^\infty \beta_k(a)\Phi_{\xi}^k(z)(a)dad\xi\right. \vspace{0.2cm}\\
    &&+\displaystyle\int_0^s \dfrac{\mu_S}{\Phi_\xi^S(z)}(\Phi_\xi^S(z)-S^*_{\sigma_1})^2 d\xi+|\J^*|\int_0^s g\left(\dfrac{S^*_{\sigma_1}}{\Phi_\xi^S(z)}\right)d\xi \\
    &&+\left.\displaystyle \sum_{k\in\J^*}\int_0^s S_{\sigma_1}^* \int_0^\infty \beta_k(a)x_{k,\alpha_k}^*(a)g\left(\dfrac{\Phi_\xi^k(v)(a)\int_0^\infty \beta_k(s)x_{k,\alpha_k}^*(s)ds}{x_{k,\alpha_k}^*(a)\int_0^\infty \beta_k(s)\Phi_\xi^k(v)(s)ds}\right)da d\xi\right| \vspace{0.2cm}\\
    &\leq& \displaystyle\left|L_1^{\alpha_1,...,\alpha_{\sigma_1}}(\Phi_{s}(z))-L_1^{\alpha_1,...,\alpha_{\sigma_1}}(\Phi_{s}(w^{(m)}))\right|+\left|L_1^{\alpha_1,...,\alpha_{\sigma_1}}(z)-L_1^{\alpha_1,...,\alpha_{\sigma_1}}(w^{(m)})\right|\\
    &&+\displaystyle\sum_{k=1+\sigma_1}^n \left(\dfrac{1}{r_k}-S^*_{\sigma_1}\right)s\|\beta_k\|_{L^\infty}\times \sup_{\xi\in[0,s]}\|\Phi^k_{\xi}(z)-\Phi^k_{\xi}(w^{(m)})\|_{L^1(\R_+)}\\
    &&+s \mu_S\times \sup_{\xi\in [0,s]} \left| \dfrac{\left(\Phi^S_{\xi}(w^{(m)})-S_{\sigma_1}^*\right)^2}{\Phi_{\xi}^S(w^{(m)})}- \dfrac{\left(\Phi^S_{\xi}(z)-S_{\sigma_1}^*\right)^2}{\Phi_{\xi}^S(z)}\right| \\
    &&+\displaystyle s |\J^*| \sup_{\xi\in[0,s]}\left|g\left(\dfrac{S_{\sigma_1}^*}{\Phi_{\xi}^S(z)}\right)-g\left(\dfrac{S_{\sigma_1}^*}{\Phi_{\xi}^S(w^{(m)})}\right)\right| \\
    &&    
    +\displaystyle s S_{\sigma_1}^*\sup_{\xi\in[0,s]}\sum_{k\in \J^*}  \int_0^\infty \beta_k(a)x^*_{k,\alpha_k}(a) \left|g\left(\dfrac{\Phi_{\xi}^k(z)(a)\int_0^\infty \beta_k(\zeta)x^*_{k,\alpha_k}(\zeta)d\zeta}{x^*_{k,\alpha_k}(a)\int_0^\infty \beta_k(\zeta)\Phi^k_{\xi}(z)(\zeta)d\zeta}\right)\right. \\
    &&\hspace{5.5cm}\left.-g\left(\dfrac{\Phi_{\xi}^k(w^{(m)})(a)\int_0^\infty \beta_k(\zeta)x^*_{k,\alpha_k}(\zeta)d\zeta}{x^*_{k,\alpha_k}(a)\int_0^\infty \beta_k(\zeta)\Phi^k_{\xi}(w^{(m)})(\zeta)d\zeta}\right)\right|da.
    \end{array}
    \end{equation}
    From \eqref{Eq:cw} and \eqref{Eq:estimate_phi_S} we get
   \begin{flalign*}
   \left|S_{\sigma_1}^*g\left(\dfrac{\Phi_s^S(w^{(m)})}{S_{\sigma_1}^*}\right)-S_{\sigma_1}^*g\left(\dfrac{\Phi_s^S(z)}{S_{\sigma_1}^*}\right)\right|&\leq \left|\Phi_s^S (w^{(m)})-\Phi_s^S(z)\right|\times\left(1+\dfrac{S^*_{\sigma_1}}{c_S-\ep}\right) \\
   &\leq \ep\times\left(1+\dfrac{S^*_{\sigma_1}}{c_S-\ep}\right) =:\kappa_1(\ep)
   \end{flalign*}
    and
    \begin{flalign*}
    &\sum_{k\in \llbracket 1+\sigma_1,n\rrbracket \cup (J\setminus \J^*)}\left|\int_0^\infty \Psi_k(a)\Phi_s^k(z)(a)da-\int_0^\infty \Psi_k(a)\Phi_s^k(w^{(m)})(a)da\right|\\
    &\leq \sum_{k\in \llbracket 1+\sigma_1,n\rrbracket \cup (J\setminus \J^*)} \|\Psi_k\|_{L^\infty}\|\Phi_s^k(z)-\Phi_s^k(w^{(m)})\|_{L^1}\leq \ep   \sum_{k\in \llbracket 1+\sigma_1,n\rrbracket \cup (J\setminus \J^*)} \|\Psi_k\|_{L^\infty}=:\kappa_2(\ep)
    \end{flalign*}
    for every $s\in[0,t]$. Using \eqref{Eq:cs_cx}-\eqref{Eq:cw}, we can define $c=\min_{k\in \J}\{c_w,c_Sc_k\}$ so that for every $k\in \J^*$ we have
   $$\dfrac{\Phi_s^k(z)(a)}{x_{k,\alpha_k}^*(a)}\geq \dfrac{cr_k}{\mu_S(\RR_{0,k}-1)}, \quad \dfrac{\Phi_s^k(w^{(m)})(a)}{x_{k,\alpha_k}^*(a)}\geq \dfrac{cr_k}{\mu_S(\RR_{0,k}-1)}$$
   for every $s\in[0,t]$ and every $a\geq 0$. It follows that
  \begin{flalign*}
  &\sum_{k\in \J^*}\left|\int_0^\infty \Psi_k(a)x_{k,\alpha_k}^*(a)g\left(\dfrac{\Phi_s^k(z)(a)}{x_{k,\alpha_k}^*(a)}\right)da-\int_0^\infty \Psi_k(a)x_{k,\alpha_k}^*(a)g\left(\dfrac{\Phi_s^k(w^{(m)})(a)}{x_{k,\alpha_k}^*(a)}\right)da\right|\\
  &\qquad \leq \sum_{k\in \J^*} \|\Psi_k\|_{L^\infty} \int_0^\infty x_{k,\alpha_k}^*(a)\left|\dfrac{\Phi_s^k(z)(a)}{x_{k,\alpha_k}^*(a)}-\dfrac{\Phi_s^k(w^{(m)})(a)}{x_{k,\alpha_k}^*(a)}\right|\times \left(1+\dfrac{\mu_S(\RR_{0,k}-1)}{cr_k}\right)da\\
  &\qquad \leq \sum_{k\in \J^*} \|\Psi_k\|_{L^\infty}\left(1+\dfrac{\mu_S(\RR_{0,k}-1)}{cr_k}\right)\|\Phi_s^k(z)-\Phi_s^k(w^{(m)})\|_{L^1} \\
  &\qquad \leq \|\Psi_k\|_{L^\infty}\left(1+\dfrac{\mu_S(\RR_{0,k}-1)}{cr_k}\right)\ep=:\kappa_3(\ep)
  \end{flalign*}
  for every $s\in[0,t]$. Furthermore, from \eqref{Eq:cw}-\eqref{Eq:estimate_phi_S} we deduce that
  \begin{flalign*}
  &\displaystyle\sum_{k=1+\sigma_1}^n \|\beta_k\|_{L^\infty}\left(\dfrac{1}{r_k}-S^*_{\sigma_1}\right)s\times \sup_{\xi\in[0,s]}\|\Phi^k_{\xi}(z)-\Phi^k_{\xi}(w^{(m)})\|_{L^1(\R_+)}\\
  &\qquad \leq \sum_{k=1+\sigma_1}^n \|\beta_k\|_{L^\infty}\left(\dfrac{1}{r_k}-S^*_{\sigma_1}\right)t\times \ep=:\kappa_4(\ep),
  \end{flalign*}
  \begin{equation*}
  \begin{array}{rcl}
   && s S_{\sigma_1}^*|\J^*|  \sup_{\xi\in[0,s]}\left|g\left(\dfrac{S_{\sigma_1}^*}{\Phi_{\xi}^S(z)}\right)-g\left(\dfrac{S_{\sigma_1}^*}{\Phi_{\xi}^S(w^{(m)})}\right)\right|\\
   &&\qquad \leq t (S^*_{\sigma_1})^2|\J^*|\left|\dfrac{1}{\Phi_{\xi}^S(z)}-\dfrac{1}{\Phi_{\xi}^S(w^{(m)}}\right|\times \left(1+\dfrac{\Lambda}{\mu_0 S^*_{\sigma_1}}\right) \\
    &&\qquad \leq  t \ep (S^*_{\sigma_1})^2|\J^*|\times \left(\dfrac{1}{c_S(c_S-\ep)}\right)\times \left(1+\dfrac{\Lambda}{\mu_0 S^*_{\sigma_1}}\right)=:\kappa_5(\ep).
  \end{array}
  \end{equation*}
 and
    \begin{flalign*}
        &s \mu_S\times \underset{\xi\in [0,s]}{\sup} \left| \dfrac{\left(\Phi^S_{\xi}(w^{(m)})-S_{\sigma_1}^*\right)^2}{\Phi_{\xi}^S(w^{(m)})}- \dfrac{\left(\Phi^S_{\xi}(z)-S_{\sigma_1}^*\right)^2}{\Phi_{\xi}^S(z)}\right| \vspace{0.2cm}\\
        &\ \leq t \mu_S\times  \underset{\xi\in [0,s]}{\sup}\left(  \dfrac{\left|\left(\Phi^S_{\xi}(w^{(m)})-S_{\sigma_1}^*\right)^2-\left(\Phi^S_{\xi}(z)-S_{\sigma_1}^*\right)^2\right|}{\Phi_{\xi}^S(z)}+\left(\Phi^S_{\xi}(w^{(m)})-S_{\sigma_1}^*\right)^2\left|\dfrac{1}{\Phi_{\xi}^S(z)}-\dfrac{1}{\Phi_{\xi}^S(w^{(m)})}\right|\right) \\
        &\ \leq  t\mu_S \left(\dfrac{2S^*_{\sigma_1}\ep+\ep(\frac{2\Lambda}{\mu_0}+\ep)}{c_S}+\dfrac{\ep(\ep+\frac{2\Lambda}{\mu_0})^2}{c_S(c_S-\ep)}\right)=:\kappa_6(\ep)
        \end{flalign*}
    since $S^*_{\sigma_1}\leq \frac{\Lambda}{\mu_0}$ and by using \eqref{Eq:c_s}. Moreover using \eqref{Eq:estimate_phi_t} and the fact that $K$ is compact and invariant lead for each $k\in \J^*$ to $\int_0^\infty \beta_k(\xi)x_{k,\alpha_k}^*(\xi)d\xi=\frac{1}{S_k^*}$ and 
  $$\int_0^\infty \beta_k(\xi)\Phi_{s}^k(z)(\xi)d\xi\leq \dfrac{\Lambda\|\beta_k\|_{L^\infty}}{\mu_0}, \ \int_0^\infty \beta_k(\xi)\Phi_{s}^k(w^{(m)})(\xi)d\xi\leq \|\beta_k\|_{L^\infty}\left(\ep+\dfrac{\Lambda}{\mu_0}\right)$$
  for every $s\in[0,t]$ then using \eqref{Eq:int_c} and \eqref{Eq:int_c-wm} it follows that
  \begin{equation*}
  \begin{array}{rcl}
  &&\displaystyle s S_{\sigma_1}^*\sup_{\xi\in[0,s]}\sum_{k\in \J^*}  \int_0^\infty \beta_k(a)x^*_{k,\alpha_k}(a) \left|g\left(\dfrac{\Phi_{\xi}^k(z)(a)\int_0^\infty \beta_k(\zeta)x^*_{k,\alpha_k}(\zeta)d\zeta}{x^*_{k,\alpha_k}(a)\int_0^\infty \beta_k(\zeta)\Phi^k_{\xi}(z)(\zeta)d\zeta}\right)\right.\\
  &&\hspace{5.5cm}\left.-g\left(\dfrac{\Phi_{\xi}^k(w^{(m)})(a)\int_0^\infty \beta_k(\zeta)x^*_{k,\alpha_k}(\zeta)d\zeta}{x^*_{k,\alpha_k}(a)\int_0^\infty \beta_k(\zeta)\Phi^k_{\xi}(w^{(m)})(\zeta)d\zeta}\right)\right|da\\
  &\leq& \displaystyle \underset{\xi\in[0,s]}{\sup}\sum_{k\in\J^*}\int_0^\infty\beta_k(a)x^*_{k,\alpha_k}\left|\dfrac{\Phi_{\xi}^k(w^{(m)})(a)\int_0^\infty \beta_k(\zeta)x_{k,\alpha_k}^*(\zeta)d\zeta}{x_{k,\alpha_k}^*(a)\int_0^\infty \beta_k(\zeta)\Phi_{\xi}^k(w^{(m)})(\zeta)d\zeta}-\dfrac{\Phi_{\xi}^k(z)(a)\int_0^\infty \beta_k(\zeta)x_{k,\alpha_k}^*(\zeta)d\zeta}{x_{k,\alpha_k}^*(a)\int_0^\infty \beta_k(\zeta)\Phi_{\xi}^k(z)(\zeta)d\zeta}\right|da \\
  &&\times \displaystyle t S^*_{\sigma_1}\left(1+\dfrac{(\Lambda+\ep\mu_0)\|\beta_k\|_{L^\infty}\mu_S (\RR_{0,k}-1)S_1^*\alpha_{k}}{c r_k \mu_0}\right) \vspace{0.1cm} \\
  &\leq& \displaystyle\sup_{\xi\in [0,s]}\sum_{k\in\J^*}\left(\int_0^\infty \beta_k(\zeta)x^*_{k,\alpha_k}(\zeta)d\zeta \right) \int_0^\infty \beta_k \left|\dfrac{\Phi_{\xi}^k(w^{(m)})(a)}{\int_0^\infty \beta_k(\zeta)\Phi_{\xi}^k(w^{(m)})(\zeta)d\zeta}-\dfrac{\Phi_{\xi}^k(z)(a)}{\int_0^\infty \beta_k(\zeta)\Phi_{\xi}^k(z)(\zeta)d\zeta}\right|da \\
  &&\times\displaystyle tS^*_{\sigma_1} \left(1+\dfrac{(\Lambda+\ep\mu_0)\|\beta_k\|_{L^\infty}\mu_S (\RR_{0,k}-1)S_1^*\alpha_{k}}{c r_k \mu_0}\right)  \\
  &\leq&\displaystyle t \sum_{k\in\J^*}\left(\dfrac{\|\ep\beta_k\|_{L^\infty}}{c_k}+\dfrac{\Lambda\ep\|\beta_k\|^2_{L^\infty}}{\mu_0 c_k(c_k-\ep\|\beta_k\|_{L^\infty})}\right) \left(1+\dfrac{(\Lambda+\ep\mu_0)\|\beta_k\|_{L^\infty}\mu_S (\RR_{0,k}-1)S_1^*\alpha_{k}}{c r_k \mu_0}\right)=:\kappa_7(\ep).
  \end{array}
  \end{equation*}
 Finally, from \eqref{Eq:L1_deriv-z} we deduce that
    \begin{equation*}
    \begin{array}{rcl}
    \H_s(z)\leq 2\displaystyle\sum_{j=1}^7 \kappa_j(\ep)\underset{\ep \to 0}{\longrightarrow} 0
    \end{array}
    \end{equation*}
    uniformly in $s\in[0,t]$. We thus get $\H_s(z)=0$ and then \eqref{Eq:Lk_int_m} is satisfied for any $s\in [0,t]$. Also, since $t\geq 0$ was taken arbitrarily then \eqref{Eq:L1_deriv-inf}-\eqref{Eq:Lk_int} hold for any $t\geq 0$ whence $L_1^{\alpha_1,...,\alpha_{\sigma_1}}$ is a Lyapunov function on every complete orbit $\gamma(z)$ for any $z\in K$ which proves the second point.
\end{enumerate}
This proves the second part of the second item of Theorem \ref{Thm:GAS}.

\subsubsection{Computation of the global attractor and conclusion}

Here we prove that $\A_0$ is reduced to the set $\E_{1}$. We first need to state the following result.

\begin{lemma}\label{Lemma:L_1-Lyap}
Let $\J\subset \llbracket 1,\sigma_1\rrbracket$ with $\J\neq \emptyset$. Let
$$z\in\left(\bigcap_{j\in \J} \S_j\right)\cap \left(\bigcap_{j\in \llbracket 1,\sigma_1\rrbracket\setminus \J} \partial \S_j\right)$$
and let
$$\gamma(z)=\{\phi(t),t\in\R\}\subset  \left(\bigcap_{j\in \J} \S_j\right)\cap \left(\bigcap_{j\in \llbracket 1,\sigma_1\rrbracket\setminus \J} \partial \S_j\right)$$
be a complete orbit through $z$. Let $(\alpha_{1},...,\alpha_{\sigma_1})\in[0,1]^{\sigma_1}$ such that $\sum_{j=1}^{\sigma_1}\alpha_j=1$ and with $\alpha_j=0$ for every $j\in\llbracket 1,\sigma_1\rrbracket \setminus \J$ and $\alpha_j>0$ for every $j\in\J$. Then the functional $L_k^{\alpha_{1},...,\alpha_{\sigma_1}}$ defined by \eqref{Eq:Lk} is a Lyapunov functional on the complete orbit $\gamma(z)$ and for each $(t,v)\in \R_+^*\times \gamma(z)$, the time derivative of $t\longmapsto L_k^{\alpha_{1},...,\alpha_{\sigma_1}}(\Phi_t(v))$ is given by \eqref{Eq:L1_deriv-inf}, with $\J^*=\J$.
If moreover $L_k^{\alpha_{1},...,\alpha_{\sigma_1}}$
is constant on $\gamma(z)$ then $\gamma(v)\subset \E_{1,\J}$ that is defined by \eqref{Eq:equilibria_kJ}. In particular $v\in \E_{1,\J}$.
\end{lemma}

\begin{proof}
The first part of the Lemma simply results from Section \ref{Sec:Lyap-functional}.   For the second part (when the Lyapunov functional is constant on $\gamma(z)$), the computations are similar to those in the proof of Lemma \ref{Lemma:L_k-Lyap} (second item). 
\end{proof}

We first know that $\E_{1}\subset \A_0$ since $\E_1$ is an invariant set that belongs to $\A_0$. Let $z\in \A_0$. By invariance of $\A_0$, let $\gamma(z)=\{\phi(t),t\in\R\}\subset \A_0$ be a complete orbit through $z$. We denote by $\J\subset \llbracket 1,\sigma_{1}\rrbracket$ the (non-empty) set such that
$$z\in \left(\bigcap_{k\in \J}\S_k\right)\cap\left( \bigcap_{k\in \llbracket1,\sigma_1\rrbracket \setminus \J}\partial \S_k\right)=:K.$$
Then we have $\phi(t)\in K$ for every $t\in \R$ by positive invariance of each set $\S_k$ and $\partial \S_k$. Let $(\alpha_{1},...,\alpha_{\sigma_1})\in[0,1]^{\sigma_1}$ such that $\sum_{j=1}^{\sigma_1}\alpha_j=1$ and with $\alpha_j=0$ for every $j\in\llbracket 1,\sigma_1\rrbracket \setminus \J$ and $\alpha_j>0$ for every $j\in\J$. From Lemma \ref{Lemma:L_1-Lyap}, we know that $L_1^{\alpha_1,...,\alpha_{\sigma_1}}$ defined by \eqref{Eq:Lk} is a Lyapunov functional on $\gamma(z)$. Using \cite[Proposition 2.51 p. 53]{SmithThieme2011} we deduce that the function $L_1^{\alpha_1,...,\alpha_{\sigma_1}}$ is constant on $\alpha(\phi)$ hence on each complete orbit of $\alpha(\phi)$. By Lemma \ref{Lemma:L_1-Lyap}, we deduce that $\alpha(\phi)\subset \E_{1,\J}$. Similarly we get $\omega(z)\subset \E_{1,\J}$. It follows that
$$\lim_{t\to -\infty}d(\phi(t),\E_{1,\J})=\lim_{t\to+\infty}d(\phi(t),\E_{1,\J})=0$$
with $\lim_{t\to -\infty}L_1^{\alpha_1,...,\alpha_{\sigma_1}}(\phi(t))\geq L_1^{\alpha_1,...,\alpha_{\sigma_1}}(z)\geq \lim_{t\to +\infty}L_1^{\alpha_1,...,\alpha_{\sigma_1}}(\phi(t))$.
Now, we apply the same arguments than those used in Section \ref{Sec:existence_A0}, point 3. (ii). We just need to consider $k=1$ and use Lemma \ref{Lemma:L_1-Lyap} instead of Lemma \ref{Lemma:L_k-Lyap}. This allow us to first prove the existence of $(E_{\tau_1,...,\tau_n}^{*,1},E^{*,1}_{\omega_1,...,\omega_n})\in(\E_{1,\J})^2$ such that \eqref{Eq:Phi_limit} holds (see Step 3), so that $\gamma(z)$ is either a heteroclinic or homoclinic orbit. Then we prove that 
$$E^{*,1}_{\tau_1,...,\tau_n}=E^{*,1}_{\omega_1,...,\omega_n}$$
(see Step 4) so that $\gamma(z)$ is actually a homoclinic orbit. Recalling that $L_1^{\alpha_1,...,\alpha_{\sigma_1}}$ is a Lyapunov functional on $\gamma(z)$, we deduce that $L_1^{\alpha_1,...,\alpha_{\sigma_1}}$ is constant on $\gamma(z)$ and by Lemma \ref{Lemma:L_1-Lyap} we get $\gamma(z)\subset \E_{1,\J}$ and in particular $z\in \E_1$. In conclusion $\A_0=\E_1$ which proves that the set $\E_1$ is GAS in $\S_0=\cup_{k=1}^{\sigma_1}\S_k$.

This proves the first statement of the second point of Theorem \ref{Thm:GAS} for $k=1$. The rest of the second point readily follows since each set $\partial \S_j$ is positively invariant. Consequently Theorem \ref{Thm:GAS} is proved.

\bibliographystyle{abbrv}
\bibliography{Biblio.bib}
\end{document}